\documentclass[11pt,oneside,reqno]{amsart}

\usepackage{amsmath,amssymb,amsthm,textcomp,mathtools}
\usepackage{amsfonts,graphicx}
\usepackage[mathscr]{eucal}
\usepackage{color}
\usepackage{csquotes}
\usepackage{enumitem}
\allowdisplaybreaks

\numberwithin{equation}{section}

\theoremstyle{definition}

\numberwithin{equation}{section}

\newtheorem{theorem}{\bf Theorem}[section]
\newtheorem{remark}{\bf Remark}[section]
\newtheorem{proposition}{Proposition}[section]

\newtheorem{example}{Example}[section]
\newtheorem{definition}{Definition}[section]
\newtheoremstyle
{remarkstyle}
{}
{11pt}
{}
{}
{\bfseries}
{:}
{     }
{\thmname{#1} \thmnumber{#2} }

\theoremstyle{remarkstyle}

\begin{document}
	\title{On Two Parameter Time-Changed Poisson Random Fields with Drifts}
	\author[Pradeep Vishwakarma]{Pradeep Vishwakarma}
	\address{Pradeep Vishwakarma, Department of Mathematics,
		Indian Institute of Technology Bhilai, Durg, 491002, INDIA.}
	\email{pradeepv@iitbhilai.ac.in}
	\author[Manisha Dhillon]{Manisha Dhillon}
	\address{Manisha Dhillon, Department of Mathematics,
		Indian Institute of Technology Bhilai, Durg, 491002, INDIA.}
	\email{manishadh@iitbhilai.ac.in}

	\author[Kuldeep Kumar Kataria]{Kuldeep Kumar Kataria}
	\address{Kuldeep Kumar Kataria, Department of Mathematics,
		 Indian Institute of Technology Bhilai, Durg, 491002, INDIA.}
	 \email{kuldeepk@iitbhilai.ac.in}

	
	\date{\today}	
	\maketitle
\begin{abstract}
We study the composition of bivariate L\'evy process with bivariate inverse subordinator. The explicit expressions for its dispersion and auto correlation matrices are obtained. Also, the time-changed two parameter L\'evy processes with rectangular increments are studied.  We introduce  some time-changed variants of the Poisson random field in plane with and without drift, and derive the associated fractional differential equations for their distributions. Later, we consider some time-changed L\'evy processes where the time-changing components are two parameter Poisson random fields with drifts. Moreover, two parameter coordinatewise semigroup operators associated with some of the introduced processes are discussed.
\end{abstract}
	
\section{Introduction} \label{sec1.1}
In the past few decades, the time-changed one parameter L\'evy processes have been extensively studied with various applications in different fields (see \cite{Meerschaert2011}, \cite{Meerschaert2012}, \cite{Straka2011}). Due to their connection with fractional operators and semi-Markov dynamics, these processes are used to study the anomalous diffusion in complex systems (see \cite{Benson2000}, \cite{Magdziarz2011}, \cite{Magdziarz2015}). A time-changed L\'evy process with inverse stable subordinator as time changing component is a scaling limit of a continuous-time random walk (see \cite{Meerschaert2008}).
 
Multiparameter L\'evy processes are natural extension of the  one parameter L\'evy processes and are of interest in the field of analysis because their densities appear as solutions of many differential equations involving multiparameter fractional operators. In \cite{Khoshnevisan2002a}, the intersection of multiparameter Markov and L\'evy processes with other branches of mathematics, such as, functional analysis, analytic number theory, \textit{etc}, are discussed. 

The Poisson random field (PRF) on $\mathbb{R}^2$ is a two parameter point process which represents the random arrangement  of points in the plane. It is a statistical model for analyzing the pattern of points which represent the spatial locations of different objects. It has potential applications in various science and engineering fields such as astronomy, ecology and telecommunication, \textit{etc}.

Let $\mathcal{B}$ be the Borel sigma algebra on $\mathbb{R}^d$, $d\ge1$, and let $|B|$ denote the Lebesgue measure of set $B\in\mathcal{B}$. The PRF on $\mathbb{R}^d$ is an integer valued random measure $\{\mathscr{N}(B),\ B\in\mathcal{B}\}$, where $\mathscr{N}(B)$ denotes the random number of points inside a Borel set $B$. It is also called the spatial Poisson point process. For disjoint sets $B_i\in\mathcal{B}$, $i=1,2,\dots,m$, the random variables $\mathscr{N}(B_1)$, $\mathscr{N}(B_2),\dots,\mathscr{N}(B_m)$ are mutually independent. Also, for $B\subset{\mathbb{R}^d}$ with $|B|<\infty$, there exist $\lambda>0$, known as the parameter of PRF, such that the random points count $N(B)$ in set $B$ has Poisson distribution with mean $\lambda|B|$, that is, (see \cite{Stoyen1995})
\begin{equation*}\label{dprfdist}
	\mathrm{Pr}\{\mathscr{N}(B)=k\}=\frac{e^{-\lambda|B|}(\lambda|B|)^k}{k!},\ k\ge0.
\end{equation*}
 The void probability of PRF is $\mathrm{Pr}\{\mathscr{N}(B)=0\}=e^{-\lambda|B|}$.

Let $\mathbb{R}_+=[0,\infty)$ denote the set of non-negative real numbers, and let us consider the PRF $\{\mathscr{N}(t_1,t_2), (t_1,t_2)\in\mathbb{R}^2_+\}$, where
$
\mathscr{N}(t_1,t_2)=\mathscr{N}([0,t_1]\times[0,t_2])$.
Here, $\mathscr{N}(t_1,t_2)$ denotes the total number of random points inside the rectangle $[0,t_1]\times[0,t_2]$. For $(s_1,s_2)\prec (t_1,t_2)$, that is, $s_1\leq t_1$ and $s_2\leq t_2$, the increment of PRF on rectangle $(s_1,t_1]\times (s_2,t_2]$ is defined by (see \cite{Merzbach1986})
\begin{equation*}
	\mathscr{N}((s_1,t_1]\times(s_2,t_2])=\mathscr{N}(t_1,t_2)-\mathscr{N}(s_1,t_2)-\mathscr{N}(t_1,s_2)+\mathscr{N}(s_1,s_2).
\end{equation*}
The PRF $\{\mathscr{N}(t_1,t_2),\ (t_1,t_2)\in\mathbb{R}^2_+\}$ on $\mathbb{R}^2_+$ is a c\`adl\`ag random field such that $\mathscr{N}(0,t_2)=\mathscr{N}(t_1,0)=0$ almost surely. The increment $\mathscr{N}((s_1,t_1]\times(s_1,t_2])$ is a Poisson random variable with parameter $\lambda(t_1-s_1)(t_2-s_2)$. Moreover, on taking $s_1=s_2=0$, we get the distribution $p(k,t_1,t_2)=\mathrm{Pr}(\mathscr{N}(t_1,t_2)=k)$ of PRF as follows:
\begin{equation}\label{prfdist}
	p(k,t_1,t_2)=\frac{e^{-\lambda t_1t_2}(\lambda t_1t_2)^k}{k!},\ k\ge0.
\end{equation}

In \cite{Beghin2014}, the composition of independent one parameter stable subordinator and inverse stable subordinator has been studied. Also, they introduced and studied various fractional Poisson processes with random drifts, and their composition with an independent L\'evy process. Recently, in \cite{Khandakar2025}, a generalized counting process with random drift has been studied.

The time-changed multivariate L\'evy processes are useful in the modeling of sub-diffusion dynamics in anisotropic medium where the trapping effect can be different for each coordinate directions (see \cite{Beghin2020}, \cite{Iafrate2024}).
In the past two decades, the time-changed multiparameter L\'evy processes have been studied by many researcher (see \cite{Barndorff2001}, \cite{Beghin2020}, \cite{Pedersen2003}-\cite{Pedersen2004b}). In \cite{Beghin2020}, the composition of bivariate Markov processes with bivariate inverse subordinator are studied.

In this paper, we study the composition of a bivariate subordinator and two parameter bivariate inverse subordinator. Its double Laplace transform with respect to space and time variables is obtained, and the governing differential equation for its density is derived. Also, we study the time-changed bivariate and two parameter L\'evy  processes. The explicit expressions for their mean vector, dispersion and auto correlation matrices are derived. 
Moreover, we study some time-changed variants of the PRF with and without drifts. It is observed that these processes reduce to some known fractional variants of the PRF as particular cases. The coordinatewise two parameter semigroup operators associated with  some of these processes are discussed.

This study is motivated by many well established theories on time-changed one parameter and multiparameter random processes, for some recent study in this direction, we refer the reader to \cite{Beghin2020}, \cite{DeGregorio2024}, \cite{Iafrate2024}, \cite{Leonenko2014}, \cite{Vishwakarma2025}, and references therein.

The outline of this paper is as follows:

In Section \ref{sec2}, first, we collect some results on one parameter stable subordinator, inverse stable subordinator and their composition. Thereafter, we study the composition of bivariate two parameter subordinator and inverse subordinator. Also, the explicit expression for the auto covariance matrix of a time-changed bivariate L\'evy process  is obtained.
In Section \ref{sec3}, we study some time-changed variants of the PRF. Their distribution and associated governing differential equations are derived.
In Section \ref{sec4}, we study various Poisson random fields with random drifts. The explicit expressions of their densities
are obtained. Also, the associated two parameter semigroup operators are discussed.
In Section \ref{sec5}, some time-changed L\'evy processes are studied where the time-changing components are PRF with and without drifts. Their stationary increments property is established, and some of their characterizations are obtained.

\section{Bivariate subordinator and its inverse}\label{sec2}	Here, we derive some results on the composition of two parameter bivariate subordinator and inverse subordinator. First, we collect some known results on the stable, inverse stable subordinators and their composition that will be used. 

 A subordinator $\{S^{\sigma}(t),\ t\ge0\}$ is a non-decreasing L\'evy process with continuous sample path (see \cite{Applebaum2004}). It is characterized by the following Laplace transform: $\mathbb{E}e^{-\eta S^\sigma(t)}=e^{-t\sigma(\eta)}$, $\eta>0$, where the Bern\v stein function $\sigma(\eta)=\int_{0}^{\infty}(1-e^{-\eta s})\,\nu(\mathrm{d}s)$ is known as the Laplace exponent of $S^\sigma(t)$. Here, $\nu$ is the L\'evy measure, that is, it satisfies $\nu(-\infty,0)=0$ and $\int_{0}^{\infty}\min\{1,s\}\nu(\mathrm{d}s)<\infty$.
 
 Let $g(x,t)$, $x\ge0$, $t\ge0$ be the density of $S^\sigma(t)$. Then, it solves the following differential equation (see \cite{Kolokoltsov2009}):
 \begin{equation}\label{subgeq}
 	\partial_tg(x,t)=-\mathcal{D}_xg(x,t),
 \end{equation}
 where the operator $\mathcal{D}_x$ is the generalized fractional derivative defined as 
 \begin{equation}\label{gfder}
 	\mathcal{D}_xf(x)=\int_{0}^{\infty}(f(x)-f(x-y))\,\nu(\mathrm{d}y).
 \end{equation}
 
 The right-continuous inverse $\{L^\sigma(t),\ t\ge0\}$ of subordinator $\{S^\sigma(t),\ t\ge0\}$ is known as the inverse subordinator. It is defined as 
 \begin{equation}\label{insubdef}
 	L^\sigma(t)=\inf\{u>0:S^\sigma(u)>t\}.
 \end{equation}
 Let $f(x,t)$, $x\ge0$, $t\ge0$ be the density of $L^\sigma(t)$. Then, it solves the following differential equation:
 \begin{equation}\label{insubequ}
 	\mathcal{D}_tf(x,t)=-\partial_xf(x,t),
 \end{equation}
 with $f(x,0)=\delta(x)$ and $f(0,t)=\int_{t}^{\infty}\nu(\mathrm{d}s)$.
 
 Let us consider the following time-changed process:
 \begin{equation}\label{compo}
 	H^{\sigma,\rho}(t)=S^\sigma(L^\rho(t)),\ t\ge0,
 \end{equation}
 where the subordinator $\{S^\sigma(t),\ t\ge0\}$ is independent of the inverse subordinator $\{L^\rho(t),\ t\ge0\}$.

Let $h(x,t)$, $x\ge0$, $t\ge0$ be the density of  (\ref{compo}). Then, it is given by
\begin{equation}\label{compoden}
	h(x,t)=\int_{0}^{\infty}g(x,s)f(s,t)\,\mathrm{d}s.
\end{equation}

Next, we derive the governing equation for the density (\ref{compoden}).
\begin{proposition}
	The density $h(x,t)=\mathrm{Pr}\{H^{\sigma,\rho}(t)\in\mathrm{d}x\}/\mathrm{d}x$ solves the following differential equation:
	\begin{equation*}
		\mathcal{D}_t h(x,t)=-\mathcal{D}_x h(x,t)+h(x,0)\int_{t}^{\infty}\nu(\mathrm{d}s),
	\end{equation*}
	where the operator $\mathcal{D}_t$ is defined in (\ref{gfder}). Its double Laplace transform is given by
	\begin{equation*}
		\int_{0}^{\infty}e^{-zt}\mathbb{E}e^{-\eta H^{\sigma,\rho}(t)}\,\mathrm{d}t=\frac{\rho(z)}{z(\sigma(\eta)+\rho(z))},\ z>0,\ \eta>0.
	\end{equation*}
\end{proposition}
\begin{proof}
	On applying the operator $\mathcal{D}_t$ on both sides of (\ref{compoden}), and using the fact that it commutes with integral, we have
	\begin{align*}
		\mathcal{D}_t h(x,t)&=\int_{0}^{\infty}g(x,s)\mathcal{D}_t f(s,t)\,\mathrm{d}s\\
		&=-\int_{0}^{\infty}g(x,s)\partial_sf(s,t)\,\mathrm{d}s,\ \ (\text{using (\ref{insubequ})}),\\
		&=-\int_{0}^{\infty}\partial_s(g(x,s)f(s,t))\,\mathrm{d}s+\int_{0}^{\infty}\partial_sg(x,s)f(s,t)\,\mathrm{d}s\\
		&=g(x,0)f(0,t)-\int_{0}^{\infty}\mathcal{D}_x g(x,s)f(s,t)\,\mathrm{d}s,\ \ (\text{using (\ref{subgeq})}),\\
		&=g(x,0)\int_{t}^{\infty}\nu(\mathrm{d}s)-\mathcal{D}_x h(x,t).
	\end{align*}
	On using the fact that $g(x,0)=h(x,0)$, we get the required governing equation.
	
	The Laplace transform of $H^{\sigma,\rho}(t)$ is
	\begin{align*}
		\int_{0}^{\infty}e^{-zt}\mathbb{E}e^{-\eta H^{\sigma,\rho}(t)}\,\mathrm{d}t&=\int_{0}^{\infty}\mathbb{E}e^{-\eta S^\sigma(x)}\int_{0}^{\infty}e^{-zt}f(x,t)\,\mathrm{d}t\,\mathrm{d}x\\
		&=\frac{\rho(z)}{z}\int_{0}^{\infty}e^{-x(\sigma(\eta)+\rho(z))}\,\mathrm{d}x=\frac{\rho(z)}{z(\sigma(\eta)+\rho(z))},\ z>0,\ \eta>0, \\
	\end{align*}
	where we have used that $\int_{0}^{\infty}e^{-zt}f(x,t)\,\mathrm{d}t=e^{-x\rho(z)}\rho(z)/z$, $z>0$ (see \cite{Beghin2020}, \cite{Kolokoltsov2009}). This completes the proof.
\end{proof}
 
\subsection{Stable subordinator and its inverse}  A subordinator $\{S_\alpha(t),\ t\ge0\}$, $\alpha\in(0,1)$ is called $\alpha$-stable subordinator if its Laplace transform is $\mathbb{E}e^{-\eta S_\alpha(t)}=e^{-t\eta^\alpha}$, $\eta>0$ (see \cite{Applebaum2004}). 
Its first passage time process $\{L_\alpha(t),\ t\ge0\}$ defined as $L_\alpha(t)\coloneqq\inf\{s>0:S_\alpha(s)>t\},\ t>0$ is called an inverse $\alpha$-stable subordinator. It is a non-decreasing process with continuous sample path such that $L_\alpha(0)=0$ with probability one.
Note that $\{S_\alpha(t),\ t\ge0\}$ and $\{L_\alpha(t),\ t\ge0\}$ are inverse processes, that is, $\mathrm{Pr}\{L_\alpha(t)<x\}=\mathrm{Pr}\{S_\alpha(x)>t\}$. So, the density function $f_\alpha(\cdot,t)$ of $L_\alpha(t)$ is given by
$f_\alpha(x,t)=-\partial_x\int_{0}^{t}g_\alpha(w,x)\,\mathrm{d}w,\ x>0$, 
where $g_\alpha(\cdot,x)$ is the density of $S_\alpha(x)$.
It is given by $f_\alpha(x,t)=t^{-\alpha}W_{-\alpha,1-\alpha}(-xt^{-\alpha})$, $x\ge0$, $t>0$. Also, the density function of $S_\alpha(t)$ is $g_\alpha(x,t)=\alpha tx^{-1}f_\alpha(t,x)$ (see \cite{Beghin2014}), where $W_{\sigma,\rho}(\cdot)$ is the Wright function defined by (see \cite{Kilbas2006}):
\begin{equation*}
	W_{\sigma,\rho}(x)=\sum_{k=0}^{\infty}\frac{x^k}{\Gamma(k\sigma+\rho)k!},\ \rho>0,\, \sigma>-1,\,x\in\mathbb{R}.
\end{equation*}

The Laplace transform of $f_\alpha(x,t)$ with respect to space variable is given by $\mathbb{E}e^{-\eta L_\alpha(t)}=E_{\alpha,1}(-\eta t^\alpha)$, $\eta>0$, whereas with respect to time variable is (see \cite{Meerschaert2011})
\begin{equation}\label{inslap}
	\int_{0}^{\infty}e^{-zt}f_\alpha(x,t)\,\mathrm{d}t=z^{\alpha-1}e^{-z^\alpha x},\ z>0.
\end{equation}
Here, $E_{\alpha,1}(\cdot)$ is the one parameter Mittag-Leffler function defined as follows (see \cite{Kilbas2006}):
\begin{equation}\label{mittag1}
	E_{\alpha,1}(x)=\sum_{k=0}^{\infty}\frac{x^k}{\Gamma(k\alpha+1)},\ \alpha>0,\,x\in\mathbb{R}.
\end{equation}  

\subsubsection{Composition of stable and inverse stable subordinators} Let $\{S_\alpha(t),\ t\ge0\}$, $\alpha\in(0,1]$  and  $\{L_\beta(t),\ t\ge0\}$, $\beta\in(0,1]$ be independent $\alpha$-stable subordinator and inverse $\beta$-stable subordinator, respectively. Let us consider the following composition process studied in \cite{Beghin2014}:
\begin{equation}\label{driftdef}
	H^{\alpha,\beta}(t)=S_\alpha(L_\beta(t)),\ t\ge0.
\end{equation} 
Its density function $f_{\alpha,\beta}(x,t)=\mathrm{Pr}\{H^{\alpha,\beta}(t)\in\mathrm{d}x\}/\mathrm{d}x$ is given by 
\begin{equation}\label{driftdist}
	f_{\alpha,\beta}(x,t)=\int_{0}^{\infty}f_\alpha(x,s)g_\beta(s,t)\,\mathrm{d}s=\frac{\alpha}{x^{1+\alpha}t^\beta}\int_{0}^{\infty}sW_{-\alpha,1-\alpha}(-sx^{-\alpha})W_{-\beta,1-\beta}(-st^{-\beta})\,\mathrm{d}s,\ x>0.
\end{equation}
It solves the following fractional differential equation (see \cite{Beghin2014}):
\begin{equation}\label{driftequ}
	\mathcal{D}^\beta_tf_{\alpha,\beta}(x,t)=-\partial^\alpha_x f_{\alpha,\beta}(x,t),\ x\ge0,\, t>0,
\end{equation} 
with initial condition $f_{\alpha,\beta}(x,0)=\delta(x)$ and boundary condition $f_{\alpha,\beta}(0,t)=0$. Here, $\mathcal{D}^\beta_t$ is the Caputo fractional derivative operator defined by
\begin{equation}\label{caputo}
	\mathcal{D}^\beta_tg(t)=\partial^\beta_tg(t)-g(0^{+})\frac{t^{-\beta}}{\Gamma(1-\beta)}=\frac{1}{\Gamma(1-\beta)}\int_{0}^{t}(t-s)^{-\beta}f'(s)\,\mathrm{d}s,\ 0<\beta<1,
\end{equation}
with $\mathcal{D}_t^1=\partial_t$. And, $\partial^\alpha_x$ is the Riemann-Liouville fractional derivative defined as follows:
\begin{equation}\label{rlder}
	\partial^\alpha_xh(x)=\frac{1}{\Gamma(1-\alpha)}\partial_x\int_{0}^{x}(x-y)^{-\alpha}h(y)\,\mathrm{d}y,\ 0<\alpha<1,
\end{equation}
with $\partial_x^1=\partial_x$.

The Laplace transform of $H^{\alpha,\beta}(t)$ is given by 
\begin{equation}\label{driftlap}
	\mathbb{E}e^{-\eta H^{\alpha,\beta}(t)}=E_{\beta,1}(-t^\beta\eta^\alpha),\ \eta>0.
\end{equation}
 Its Laplace transform with respect to time variable is 
$\int_{0}^{\infty}e^{-zt}E_{\beta,1}(-t^\beta\eta^\alpha)\,\mathrm{d}t={z^{\beta-1}}/{(z^\beta+\eta^\alpha)}$, $z>0$, $\eta>0.
$
Thus, 
\begin{equation}\label{driftlap2}
	\int_{0}^{\infty}e^{-zt}f_{\alpha,\beta}(x,t)\,\mathrm{d}t=z^{\beta-1}x^{\alpha-1}E_{\alpha,\alpha}(-z^\beta x^\alpha),\ x\ge0,\,z>0,
\end{equation}
where $E_{\alpha,\beta}(\cdot)$ is the two parameter Mittag-Leffler function defined as (see \cite{Kilbas2006})
\begin{equation}
	E_{\alpha,\beta}(x)=\sum_{k=0}^{\infty}\frac{x^k}{\Gamma(k\alpha+\beta)},\ \alpha>0,\,\beta\in\mathbb{R},\,x\in\mathbb{R}.
\end{equation}

\subsection{Bivariate subordinator and its inverse} Let us recall some known results on the bivariate subordinator and its inverse, for more details, we refer the reader to \cite{Beghin2020} and \cite{Iafrate2024}. 

A bivariate L\'evy process $\{(S_1(t_1)$, $S_2(t_2)),\ (t_1,t_2)$ $\in\mathbb{R}^2_+\}$ is called the bivariate two parameter subordinator if its marginals are almost surely non-decreasing, that is, $\{S_i(t_i),\ t_i\ge0\}$, $i=1,2$ are one parameter subordinators. Its double Laplace transform is given by (see \cite{Beghin2020}, Proposition 2.4)
\begin{align}\label{bvsublap} 
	\mathbb{E}e^{-\eta_1S_1(t_1)-\eta_2S_2(t_2)}&=e^{-t_1B(\eta_1,\eta_2)}e^{-(t_2-t_1)B_2(\eta_2)}\textbf{1}_{\{t_2\ge t_1\}}\nonumber\\
	&\ \ +e^{-t_2B(\eta_1,\eta_2)}e^{-(t_1-t_2)B_1(\eta_1)}\textbf{1}_{\{t_1>t_2\}},\ \eta_1>0,\ \eta_2>0.
\end{align}
Here,
\begin{equation}\label{mbif}
	B_i(\eta_i)=\int_{0}^{\infty}(1-e^{-\eta_i x})\nu_i(\mathrm{d}x),\ i=1,2,
\end{equation}
are  Bern\v stein functions and
\begin{equation}\label{bibernsfncn}
	B(\eta_1,\eta_2)=\int\int_{\mathbb{R}^2_+}(1-e^{-\eta_1x_1-\eta_2x_2})\nu(\mathrm{d}x_1,\mathrm{d}x_2)
\end{equation}
is a bivariate Bern\v stein function,
where $\nu(\cdot,\cdot)$ is the L\'evy measure of associated bivariate subordinator such that
\begin{equation*}
	\int_{0}^{\infty} \int_{0}^{\infty}\min\Big\{1,\sqrt{x_1^2+x_2^2}\Big\}\nu(\mathrm{d}x_1,\mathrm{d}x_2)<\infty.
\end{equation*}

Let $g_i(x_i,t_i)$ be the density of $\{S_i(t_i), \ t_i\ge 0\}$, $i=1,2$. For $t_1\ne t_2$, let $g(x_1,x_2,t_1,t_2)$, $t_1>0$, $t_2>0$ be the density of $(S_1(t_1),S_2(t_2))$, and for $t_1=t_2=t$, let it be denoted by $g_*(x_1,x_2,t)$, $t>0$. The following holds true: 
\begin{equation}\label{incond1}
	\lim_{t_1\to0}g(x_1,x_2,t_1,t_2)=\delta_0(x_1)g_2(x_2,t_2)
\end{equation}
and
\begin{equation}\label{incond2}
	\lim_{t_2\to0}g(x_1,x_2,t_1,t_2)=\delta_0(x_2)g_1(x_1,t_1),
\end{equation}
with
\begin{equation}\label{inicondt0}
	\mathrm{Pr}\{S_1(0)\in\mathrm{d}x_1,S_2(0)\in\mathrm{d}x_2\}=\delta_0(\mathrm{d}x_1)\delta_0(\mathrm{d}x_2),
\end{equation}
where $\delta_0$ denotes the standard Dirac delta measure. 

The following result will be used (see \cite{Beghin2020}, Proposition 2.5):
\begin{proposition}
	For $t_1\ne t_2$, the joint density $g(x_1,x_2,t_1,t_2)$ solves
	\begin{equation}\label{subeq1}
		(\partial_{t_1}+\partial_{t_2})g(x_1,x_2,t_1,t_2)=-\mathcal{D}_{x_1,x_2}g(x_1,x_2,t_1,t_2),\ t_1>0,\,t_2>0,
	\end{equation}
	with initial conditions \eqref{incond1} and \eqref{incond2}. For $t_1=t_2=t$, we have
	\begin{equation}\label{subeq2}
		\partial_{t}g_*(x_1,x_2,t)=-\mathcal{D}_{x_1,x_2}g_*(x_1,x_2,t), \ t>0,
	\end{equation}
	with initial condition \eqref{inicondt0}. Here, the operator $\mathcal{D}_{t_1,t_2}$ is defined as follows:
	\begin{equation}\label{opt2}
		\mathcal{D}_{t_1,t_2}f(t_1,t_2)=\int_{0}^{\infty}\int_{0}^{\infty}(f(t_1,t_2)-f(t_1-s_1,t_2-s_2))\,\nu(\mathrm{d}s_1,\mathrm{d}s_2).
	\end{equation} 
\end{proposition}

Let us consider the inverse process $\{(L_1(t_1), L_2(t_2)),\ (t_1,t_2)\in\mathbb{R}^2_+\}$ of $\{(S_1(t_1),S_2(t_2)),\ (t_1,t_2)\in\mathbb{R}^2_+\}$ defined as 
\begin{equation*}
	L_i(t_i)=\inf\{u_i>0:S_i(u_i)>t_i\},\ i=1,2.
\end{equation*}
For $i=1,2$, the density function $f_i(x_i,t_i)$ of $L_i(t_i)$  solves the following generalized fractional Cauchy problem:
\begin{equation}
	\partial_{x_i}f_i(x_i,t_i)=-\mathcal{D}_{t_i}f_i(x_i,t_i),
\end{equation}
with $f_i(0,t_i)=\mathcal{D}_{t_i}\mathcal{H}(t_i)=\nu_i^*(t_i)$, where $\mathcal{H}(\cdot)$ is the Heaviside step function and $\nu_i^*(t_i)$ is the tail of L\'evy measure $\nu_i$, that is, $\nu_i^*(t_i)=\nu_i[t_i,\infty)$.

The Laplace transform of density function $h_i(x_i,t_i)$ with respect to time variable is given by (see \cite{Beghin2020})
\begin{equation}\label{lapinvbiv}
	\int_{0}^{\infty}e^{-\eta_i t_i}f_i(x_i,t_i)\,\mathrm{d}t_i=\frac{B_i(\eta_i)}{\eta_i}e^{-B_i(\eta_i)x_i}.
\end{equation}
Also, its space-time Laplace transform is 
\begin{equation}\label{doublapinvbiv}
	\int_{0}^{\infty}\int_{0}^{\infty}e^{-\eta_it_i}e^{-\theta_ix_i}f_i(x_i,t_i)\,\mathrm{d}x_i\,\mathrm{d}t_i=\frac{B_i(\eta_i)}{\eta_i(\theta_i+B_i(\eta_i))}.
\end{equation}

For $t_1>0$, $t_2>0$, the distribution of $\{(L_1(t_1),L_2(t_2)),\ (t_1,t_2)\in\mathbb{R}^2_+\}$ has an absolute continuous component with associated density
\begin{equation}\label{deninva}
	f(x_1,x_2,t_1,t_2)=\frac{\partial^2}{\partial x_1 \partial x_2}\int_{0}^{t_1}\int_{0}^{t_2}g(s_1,s_2,x_1,x_2)\, \mathrm{d}s_1\, \mathrm{d}s_2,\ x_1\ne x_2,
\end{equation}
where $g(s_1,s_2,x_1,x_2)$ is the density of $(S_1(x_1),S_2(x_2))$, $x_1>0$, $x_2>0$. Moreover, for $x_1=x_2$, the density of $\{(L_1(t_1),L_2(t_2)),\ (t_1,t_2)\in\mathbb{R}^2_+\}$ is given by
\begin{equation}\label{deninvb}
	f_*(x,t_1,t_2)=\frac{\partial^2}{\partial x^2}\int_{0}^{t_1}\int_{0}^{t_2}g_*(s_1,s_2,x)\,\mathrm{d}s_1\, \mathrm{d}s_2.
\end{equation}
The Laplace transform of density functions \eqref{deninva} and \eqref{deninvb} are
{\small\begin{align}\label{bvisublap1}
		\tilde{f}(x_1,x_2,\eta_1,\eta_2)&=\frac{1}{\eta_1\eta_2}B_2(\eta_2)(B(\eta_1,\eta_2)-B_2(\eta_2))\exp(-x_1(B(\eta_1,\eta_2)-B_2(\eta_2))-x_2B_2(\eta_2))\textbf{1}_{\{x_1<x_2\}}\nonumber\\
		&\ \ +\frac{1}{\eta_1\eta_2}B_1(\eta_1)(B(\eta_1,\eta_2)-B_1(\eta_1))\exp(-x_2(B(\eta_1,\eta_2)-B_1(\eta_1))-x_1B_1(\eta_1))\textbf{1}_{\{x_1>x_2\}}
\end{align}}
and 
\begin{equation}\label{bvisublap2}
	\tilde{f_*}(x,\eta_1,\eta_2)=\frac{B_1(\eta_1)+B_2(\eta_2)-B(\eta_1,\eta_2)}{\eta_1\eta_2}e^{-xB(\eta_1,\eta_2)},
\end{equation}
respectively.

Now, we consider the composition of bivariate subordinator and bivariate inverse subordinator, and study some of its distributional properties.

Let $\{(S_1(t_1),S_2(t_2)),\ (t_1,t_2)$ $\in\mathrm{R}^2_+\}$ be a bivariate subordinator, and let $\{(L_1(t_1),L_2(t_2))$, $ (t_1,t_2)\in\mathbb{R}^2_+\}$ be a bivariate inverse subordinator corresponding to a bivariate subordinator that is independent of $\{(S_1(t_1),S_2(t_2)),\ (t_1,t_2)\in\mathrm{R}^2_+\}$. We consider a bivariate process $\{(H_1(t_1),H_2(t_2))$, $(t_1,t_2)\in\mathbb{R}^2_+\}$ defined as follows:
\begin{equation}\label{tpcomp}
	(H_1(t_1),H_2(t_2))\coloneqq(S_1(L_1(t_1)),S_2(L_2(t_2))),\ (t_1,t_2)\in\mathbb{R}^2_+.
\end{equation}

Let $h(x_1,x_2,t_1,t_2)=\mathrm{Pr}\{H_1(t_1)\in\mathrm{d}x_1,H_2(t_2)\in\mathrm{d}x_2\}/\mathrm{d}x_1\mathrm{d}x_2$, $x_i\ge0$, $t_i\ge0$, $i=1,2$ be the joint density of (\ref{tpcomp}), and let $h_i(x_i,t_i)$ be the density of its marginal $H_i(t_i)$. Then, 
\begin{equation}\label{compden}
	h(x_1,x_2,t_1,t_2)=\int_{0}^{\infty}\int_{0}^{\infty}g(x_1,x_2,s_1,s_2)f(s_1,s_2,t_1,t_2)\,\mathrm{d}s_1\,\mathrm{d}s_2 +\int_{0}^{\infty}g_*(x_1,x_2,s)f_*(s,t_1,t_2)\,\mathrm{d}s.
\end{equation}

\begin{proposition}
	The space-time Laplace transform of (\ref{tpcomp}) is given by
	\begin{align*}
		\int_{0}^{\infty}\int_{0}^{\infty}e^{-z_1t_1-z_2t_2}&\mathbb{E}e^{-\eta_1H_1(t_1)-\eta_2H_2(t_2)}\,\mathrm{d}t_1\,\mathrm{d}t_2\\
		&=\frac{1}{z_1z_2(B(\eta_1,\eta_2)+B(z_1,z_2))}\Big(\frac{B_2(z_2)(B(z_1,z_2)-B_2(z_2))}{B_2(\eta_2)+B_2(z_2)}\\ &\ \ +\frac{B_1(z_1)(B(z_1,z_2)-B_1(z_1))}{B_1(\eta_1)+B_1(z_1)}+B_1(z_1)+B_2(z_2)-B(z_1,z_2)\Big).
	\end{align*}
\end{proposition}
\begin{proof}
	From (\ref{compden}), we have
	\begin{align*}
		\int_{0}^{\infty}\int_{0}^{\infty}&e^{-z_1t_1-z_2t_2}\mathbb{E}e^{-\eta_1H_1(t_1)-\eta_2H_2(t_2)}\,\mathrm{d}t_1\,\mathrm{d}t_2\\
		&=\int_{0}^{\infty}\int_{0}^{\infty}\mathbb{E}e^{-\eta_1S_1(x_1)-\eta_2S_2(x_2)}\int_{0}^{\infty}\int_{0}^{\infty}e^{-z_1t_1-z_2t_2}f(x_1,x_2,t_1,t_2)\,\mathrm{d}t_1\,\mathrm{d}t_2\,\mathrm{d}x_1\,\mathrm{d}x_2\\
		&\ \ +\int_{0}^{\infty}\mathbb{E}e^{-\eta_1S_1(x)-\eta_2S_2(x)}\int_{0}^{\infty}\int_{0}^{\infty}e^{-z_1t_1-z_2t_2}f_*(x,t_1,t_2)\,\mathrm{d}t_1\,\mathrm{d}t_2\,\mathrm{d}x.
	\end{align*}
	On splitting the first double integral into two regions, \textit{viz}. $x_1<x_2$ and $x_1>x_2$, and using (\ref{bvsublap}), (\ref{bvisublap1}) and (\ref{bvisublap2}), we get
	\begin{align*}
		\int_{0}^{\infty}\int_{0}^{\infty}&e^{-z_1t_1-z_2t_2}\mathbb{E}e^{-\eta_1H_1(t_1)-\eta_2H_2(t_2)}\,\mathrm{d}t_1\,\mathrm{d}t_2\\
		&=\int_{0}^{\infty}\int_{0}^{x_2}\mathbb{E}e^{-\eta_1S_1(x_1)-\eta_2S_2(x_2)}\int_{0}^{\infty}\int_{0}^{\infty}e^{-z_1t_1-z_2t_2}f(x_1,x_2,t_1,t_2)\,\mathrm{d}t_1\,\mathrm{d}t_2\,\mathrm{d}x_1\,\mathrm{d}x_2\\
		&\ \ +\int_{0}^{\infty}\int_{x_2}^{\infty}\mathbb{E}e^{-\eta_1S_1(x_1)-\eta_2S_2(x_2)}\int_{0}^{\infty}\int_{0}^{\infty}e^{-z_1t_1-z_2t_2}f(x_1,x_2,t_1,t_2)\,\mathrm{d}t_1\,\mathrm{d}t_2\,\mathrm{d}x_1\,\mathrm{d}x_2\\
		&\ \ +\int_{0}^{\infty}\mathbb{E}e^{-\eta_1S_1(x)-\eta_2S_2(x)}\int_{0}^{\infty}\int_{0}^{\infty}e^{-z_1t_1-z_2t_2}f_*(x,t_1,t_2)\,\mathrm{d}t_1\,\mathrm{d}t_2\,\mathrm{d}x\\
		&=\frac{B_2(z_2)(B(z_1,z_2)-B_2(z_2))}{z_1z_2}\int_{0}^{\infty}\int_{0}^{x_2}e^{-x_1B(\eta_1,\eta_2)}e^{-(x_2-x_1)B_2(\eta_2)}\\
		&\hspace{6cm}\cdot\exp(-x_1(B(z_1,z_2)-B_2(z_2))-x_2B_2(z_2))\,\mathrm{d}x_1\,\mathrm{d}x_2\\
		&\ \ +\frac{B_1(z_1)(B(z_1,z_2)-B_1(z_1))}{z_1z_2}\int_{0}^{\infty}\int_{x_2}^{\infty}e^{-x_2B(\eta_1,\eta_2)}e^{-(x_1-x_2)B_1(\eta_1)}\\
		&\hspace{6cm}\cdot\exp(-x_2(B(z_1,z_2)-B_1(z_1))-x_1B_1(z_1))\,\mathrm{d}x_1\,\mathrm{d}x_2\\
		&\ \ +\frac{B_1(z_1)+B_2(z_2)-B(z_1,z_2)}{z_1z_2}\int_{0}^{\infty}e^{-xB(\eta_1,\eta_2)}e^{-xB(z_1,z_2)}\,\mathrm{d}x\\
		&=\frac{B_2(z_2)(B(z_1,z_2)-B_2(z_2))}{z_1z_2(B(\eta_1,\eta_2)-B_2(\eta_2)+B(z_1,z_2)-B_2(z_2))}\\
		&\hspace{5cm}\cdot\int_{0}^{\infty}(e^{-x_2(B_2(\eta_2)+B_2(z_2))}-e^{-x_2(B(\eta_1,\eta_2)+B(z_1,z_2))})\,\mathrm{d}x_2\\
		&\ \ +\frac{B_1(z_1)(B(z_1,z_2)-B_1(z_1))}{z_1z_2(B_1(\eta_1)+B_1(z_1))}\int_{0}^{\infty}e^{-x_2(B(\eta_1,\eta_2)+B(z_1,z_2))}\,\mathrm{d}x_2\\
		&\ \ +\frac{B_1(z_1)+B_2(z_2)-B(z_1,z_2)}{z_1z_2}\int_{0}^{\infty}e^{-x(B(\eta_1,\eta_2)+B(z_1,z_2))}\,\mathrm{d}x.
	\end{align*} 
	On evaluating these integrals, we get the required result.
\end{proof}

Next, we derive the governing equation for the joint density of process defined in (\ref{tpcomp}).
\begin{proposition}
	The joint density function (\ref{compden}) satisfies the following equation:
	\begin{equation}\label{compogeq}
		\mathcal{D}_{t_1,t_2}h(x_1,x_2,t_1,t_2)=-\mathcal{D}_{x_1,x_2}h(x_1,x_2,t_1,t_2),\ x_1>0,\ x_2>0.
	\end{equation}
\end{proposition}
\begin{proof}
	On applying the operator $\mathcal{D}_{t_1,t_2}$  on both sides of (\ref{compden}) and using Theorem 3.6 of \cite{Beghin2020}, we have 
	\begin{align*}
		\mathcal{D}_{t_1,t_2}h(x_1,x_2,t_1,t_2)&=\int_{0}^{\infty}\int_{0}^{\infty}g(x_1,x_2,s_1,s_2)\mathcal{D}_{t_1,t_2}f(s_1,s_2,t_1,t_2)\,\mathrm{d}s_1\,\mathrm{d}s_2\\
		&\ \  +\int_{0}^{\infty}g_*(x_1,x_2,s)\mathcal{D}_{t_1,t_2}f_*(s,t_1,t_2)\,\mathrm{d}s\\
		&=-\int_{0}^{\infty}\int_{0}^{\infty}g(x_1,x_2,s_1,s_2)\partial_{s_1}f(s_1,s_2,t_1,t_2)\,\mathrm{d}s_1\,\mathrm{d}s_2\\
		&\ \ -\int_{0}^{\infty}\int_{0}^{\infty}g(x_1,x_2,s_1,s_2)\partial_{s_2}f(s_1,s_2,t_1,t_2)\,\mathrm{d}s_1\,\mathrm{d}s_2\\
		&\ \ -\int_{0}^{\infty}g_*(x_1,x_2,s)\partial_{s}f_*(s,t_1,t_2)\,\mathrm{d}s.
	\end{align*}
	On integrating by part, and by using (\ref{incond1}), (\ref{incond2}), (\ref{subeq1}) and (\ref{subeq2}), we get
	\begin{align*}
		\mathcal{D}_{t_1,t_2}&h(x_1,x_2,t_1,t_2)\\
		&=\int_{0}^{\infty}\int_{0}^{\infty}\partial_{s_1}g(x_1,x_2,s_1,s_2)f(s_1,s_2,t_1,t_2)\,\mathrm{d}s_1\,\mathrm{d}s_2+\delta_0(x_1)\int_{0}^{\infty}g_2(x_2,s_2)f(0,s_2,t_1,t_2)\,\mathrm{d}s_2\\
		&\ \ +\int_{0}^{\infty}\int_{0}^{\infty}\partial_{s_2}g(x_1,x_2,s_1,s_2)f(s_1,s_2,t_1,t_2)\,\mathrm{d}s_1\,\mathrm{d}s_2+\delta_0(x_2)\int_{0}^{\infty}g_1(x_1,s_1)f(s_1,0,t_1,t_2)\,\mathrm{d}s_1\\
		&\ \ +\int_{0}^{\infty}\partial_{s}g_*(x_1,x_2,s)f_*(s,t_1,t_2)\,\mathrm{d}s+\delta_0(x_1)\delta_0(x_2)\nu^*(t_1,t_2)\\
		&=-\mathcal{D}_{x_1,x_2}h(x_1,x_2,t_1,t_2)+\delta_0(x_1)\int_{0}^{\infty}g_2(x_2,s_2)f(0,s_2,t_1,t_2)\,\mathrm{d}s_2\\
		&\ \ + \delta_0(x_2)\int_{0}^{\infty}g_1(x_1,s_1)f(s_1,0,t_1,t_2)\,\mathrm{d}s_1+\delta_0(x_1)\delta_0(x_2)\nu^*(t_1,t_2).
	\end{align*}
	This completes the proof.
\end{proof}
\subsection{Time-changed bivariate L\'evy process} Beghin \textit{et al.} \cite{Beghin2020} studied a time-changed bivariate Markov process $(X_1(L_1(t_1)),X_2(L_2(t_2)))$, where $\{(X_1(t_1),X_2(t_2)),\ (t_1,t_2)\in\mathbb{R}^2_+\}$ is a bivariate Markov process whose marginals are independent and $(X_1(0),X_2(0))=(0,0)$ almost surely, and $\{(L_1(t_1),L_2(t_2)),\ (t_1,t_2)\in\mathbb{R}^2_+\}$ is a bivariate inverse subordinator that is independent of $\{(X_1(t_1),X_2(t_2)),\ (t_1,t_2)\in\mathbb{R}^2_+\}$. They derived the governing equation for the joint density function of introduced process.

Here, we study a time-changed bivariate L\'evy process. We derive the general expression for its auto correlation matrix. Let $\{(Y_1(t_1),Y_2(t_2)),\ (t_1,t_2)\in\mathbb{R}^2_+\}$ be a bivariate L\'evy process whose marginals $Y_i(t_i)$'s are independent real valued L\'evy processes and $(Y_1(0),Y_2(0))=(0,0)$ almost surely. These type of processes are also known as multiplicative L\'evy processes in the sense of \cite{Khoshnevisan2002a}.
Let us consider the following time-changed process:
\begin{equation}\label{tclp}
	(Z_1(t_1),Z_2(t_2))=(Y_1(L_1(t_1)),Y_2(L_2(t_2))),\ (t_1,t_2)\in\mathbb{R}^2_+,
\end{equation}
where $\{(L_1(t_1),L_2(t_2)),\ (t_1,t_2)\in\mathbb{R}^2_+\}$ is a bivariate inverse subordinator that is independent of $\{(Y_1(t),Y_2(t)),\ t\ge0\}$. For $j=1,2$, let $p_j(x_j,t_j)$ be the density of $Y_j(t_j)$. Then, it satisfies the equation $\partial_{t_j}p_j(x_j,t_j)=G_jp_j(x_j,t_j)$,
where $G_1$ and $G_2$ are the dual of generator of $Y_1(t_1)$ and $Y_2(t_2)$, respectively. Also, the joint density $p(x_1,x_2,t_1,t_2)=\mathrm{Pr}\{Y_1(t_1)\in\mathrm{d}x_1,Y_2(t_2)\in\mathrm{d}x_2\}/\mathrm{d}x_1\mathrm{d}x_2$ solves
\begin{equation*}
	\partial_{t_j}p(x_1,x_2,t_1,t_2)=G_jp(x_1,x_2,t_1,t_2),\ j=1,2.
\end{equation*}
So, the joint density function $q(x_1,x_2,t_1,t_2)=\mathrm{Pr}\{Z_1(t_1)\in\mathrm{d}x_1,Z_2(t_2)\in\mathrm{d}x_2\}/\mathrm{d}x_1\mathrm{d}x_2$ of time-changed process (\ref{tclp}) solves (see \cite{Iafrate2024})
\begin{equation*}
	\mathcal{D}_{t_1,t_2}q(x_1,x_2,t_1,t_2)=(G_1+G_2)q(x_1,x_2,t_1,t_2),\ x_1>0,\ x_2>0,
\end{equation*}
where the operator $\mathcal{D}_{t_1,t_2}$ is as defined in (\ref{opt2}).

\begin{example}(\textbf{Time-changed bivariate Brownian motion}) Let us consider the bivariate Brownian motion $\{(B_1(t_1),B_2(t_2)),\ (t_1,t_2)\in\mathbb{R}^2_+\}$, where its marginals $\{B_i(t),\ t\ge0\}$'s are independent $\mathbb{R}^N$-valued, $N\ge1$ one parameter Brownian motions. Now, consider the time-changed process $B^\alpha(t_1,t_2)=(B_1(L_1^\alpha(t_1)),B_2(L_2^\alpha(t_2)))$, $0<\alpha<1$, where $\{(L_1^\alpha(t_1),L_2^\alpha(t_2)),\ (t_1,t_2)\in\mathbb{R}^2_+\}$ is a bivariate inverse stable subordinator (for definition see \cite[Example 2.2]{Beghin2020} and \cite{Iafrate2024}) that is independent of $\{(B_1(t_1),B_2(t_2)),\ (t_1,t_2)\in\mathbb{R}^2_+\}$. If $t_1=t_2=t$ and $N=1$ then the one parameter time-changed process $\{(B_1(L_1^\alpha(t)),B_2(L_2^\alpha(t))), t\ge0\}$ represents the motion of particles in an anisotropic medium in which the trapping effect can be different for each coordinate direction. Also, $(B_1(L_1^\alpha(t)),B_2(L_2^\alpha(t)))$ is a scaling limit of a bivariate continuous-time random walk (see \cite[Theorem 5.2]{Beghin2020}).
	
	Note that $B^\alpha(t_1,t_2)$ is  $\mathbb{R}^{2N}$-valued process. So, for any bounded measurable function $f:\mathbb{R}^{2N}\rightarrow\mathbb{R}$, we define a two parameter operator 
	\begin{equation*}
		\mathcal{T}_{t_1,t_2}^\alpha f(x)=\mathbb{E}f(x+B^\alpha(t_1,t_2)),\ (t_1,t_2)\in\mathbb{R}^2_+,\ x\in\mathbb{R}^{2N}.
	\end{equation*}
	For $(t_1,t_2)\succ(0,0)$, let us consider a function 
	\begin{equation*}
		q(s_1,s_2,u,v)=(2\pi\sqrt{t_1t_2})^{-N}\exp\bigg(-\frac{||u||^2}{2t_1}-\frac{||v||^2}{2t_2}\bigg),\ u\in\mathbb{R}^N,\ v\in\mathbb{R}^N,
	\end{equation*}
	where $||\cdot||$ is the standard Euclidean norm on $\mathbb{R}^N$. Then, the operator $\mathcal{T}_{t_1,t_2}^\alpha$ is given by
	\begin{equation*}
		\mathcal{T}_{t_1,t_2}^\alpha f(x)=\int_{0}^{\infty}\int_{0}^{\infty}\int_{\mathbb{R}^N}\int_{\mathbb{R}^N}q(s_1,s_2,u,v)f(x+(u,v))\mathrm{Pr}\{L_1^\alpha(t_1)\in\mathrm{d}s_1,L_2^\alpha(t_2)\in\mathrm{d}s_2\}\,\mathrm{d}u\,\mathrm{d}v\,\mathrm{d}s_1\,\mathrm{d}s_2,
	\end{equation*}
	where $(u,v)\in\mathbb{R}^{2N}$ such that $u$ and $v$ are in $\mathbb{R}^N$. 
	It can be established using the definition of $B^\alpha(t_1,t_2)$ and Example 2 of Chapter 11 in \cite{Khoshnevisan2002a}.

\end{example} 

In \cite{Leonenko2014}, general expressions of the auto correlation function for time-changed L\'evy process is derived. Next, we derive the explicit expressions for the mean vector and the dispersion matrix of time-changed process (\ref{tclp}).

Let $\psi_j(\xi_j)=-\log\mathbb{E}e^{i\xi_jY_j(1)}$, $\xi_j\in\mathbb{R}$ be the characteristic exponent of L\'evy process $Y_j(1)$ for each $j=1,2$. Then, the joint characteristic function of (\ref{tclp}) is given by
\begin{align}
	\mathbb{E}e^{i\xi_1Z_1(t_1)+i\xi_2Z_2(t_2)}&=\int_{0}^{\infty}\int_{0}^{\infty}\mathbb{E}e^{i\xi_1Y_1(s_1)+i\xi_2Y_2(s_2)}\mathrm{Pr}\{L_1(t_1)\in\mathrm{d}s_1,L_2(t_2)\in\mathrm{d}s_2\}\nonumber\\
	&=\int_{0}^{\infty}\int_{0}^{\infty}e^{-s_1\psi_1(\xi_1)-s_2\psi_2(\xi_2)}\mathrm{Pr}\{L_1(t_1)\in\mathrm{d}s_1,L_2(t_2)\in\mathrm{d}s_2\}\nonumber\\
	&=\mathbb{E}e^{-L_1(t_1)\psi_1(\xi_1)-L_2(t_2)\psi_2(\xi_2)},\ (t_1,t_2)\in\mathbb{R}^2_+.\label{jcf}
\end{align}

\begin{proposition}\label{proptclp}
	Let $\{(Z_1(t_1),Z_2(t_2)),\ (t_1,t_2)\in\mathbb{R}^2_+\}$ be a process as defined in (\ref{tclp}). For $(t_1,t_2)\in\mathbb{R}^2_+$, if the first moments of $(X_1(t_1),X_2(t_2))$ and $(L_1(t_1),L_2(t_2))$ exist then the first moment of (\ref{tclp}) exists and 
	\begin{equation}\label{tclpmean}
		\mathbb{E}Z_j(t_j)=\mathbb{E}L_j(t_j)\mathbb{E}Y_j(1),\ j=1,2.
	\end{equation}
	Also, its variance-covariance matrix is given by
	\begin{equation}\label{cvm}
		\Sigma(t_1,t_2)=\begin{pmatrix}
			\mathbb{E}L_1(t_1)\mathbb{V}\mathrm{ar}Y_1(1)+\mathbb{V}\mathrm{ar}L_1(t_1)(\mathbb{E}Y_1(1))^2& \mathbb{E}Y_1(1)Y_2(1)\mathbb{C}\mathrm{ov}(L_1(t_1),L_2(t_2))\vspace{0.2cm}\\
			\mathbb{E}Y_1(1)Y_2(1)\mathbb{C}\mathrm{ov}(L_1(t_1),L_2(t_2))&\mathbb{E}L_2(t_2)\mathbb{V}\mathrm{ar}Y_2(1)+\mathbb{V}\mathrm{ar}L_2(t_2)(\mathbb{E}Y_2(1))^2
		\end{pmatrix}.
	\end{equation}
\end{proposition}
\begin{proof}
	Note that $\mathbb{E}e^{i\xi_jY_j(t_j)}=e^{-t_j\psi_j(\xi_j)}$. So, $\mathbb{E}Y_j(t_j)=-it_j\psi_j'(0)$, $j=1,2$. From (\ref{jcf}), we have
	\begin{equation*}
		\mathbb{E}Z_1(t_1)=-i\partial_{\xi_1}\mathbb{E}e^{i\xi_1Z_1(t_1)+i\xi_2Z_2(t_2)}\big|_{\xi_1=0,\,\xi_2=0}=i\mathbb{E}L_1(t_1)\psi_1'(0)=\mathbb{E}L_1(t_1)\mathbb{E}Y_1(1).
	\end{equation*}
	Similarly, we have $\mathbb{E}Z_2(t_2)=\mathbb{E}L_2(t_2)\mathbb{E}Y_2(1)$.
	
	For their joint moment, from (\ref{jcf}), we have
	\begin{equation*}
		\mathbb{E}Z_1(t_1)Z_2(t_2)=-\partial_{\xi_2}\partial_{\xi_1}\mathbb{E}e^{i\xi_1Z_1(t_1)+i\xi_2Z_2(t_2)}|_{\xi_1=0,\,\xi_2=0}=\mathbb{E}L_1(t_1)L_2(t_2)\psi_1'(0)\psi_2'(0).
	\end{equation*}
	Thus,
	\begin{align*}
		\mathbb{C}\mathrm{ov}(Z_1(t_1),Z_2(t_2))&=\mathbb{E}Z_1(t_1)Z_2(t_2)-\mathbb{E}Z_1(t_1)\mathbb{E}Z_2(t_2)\\
		&=\mathbb{E}Y_1(1)Y_2(1)(\mathbb{E}L_1(t_1)L_2(t_2)-\mathbb{E}L_1(t_1)\mathbb{E}L_2(t_2))\\
		&=\mathbb{E}Y_1(1)Y_2(1)\mathbb{C}\mathrm{ov}(L_1(t_1),L_2(t_2)).
	\end{align*}
	
	From (\ref{jcf}), the second moment of $Z_1(t_1)$ is given by
	\begin{equation*}
		\mathbb{E}Z_1(t_1)^2=-\partial^2_{\xi_1}\mathbb{E}e^{i\xi_1Z_1(t_1)+i\xi_2Z_2(t_2)}|_{\xi_1=0,\,\xi_2=0}=\mathbb{E}L_1(t_1)\psi_1{''}(0)-\mathbb{E}L_1(t_1)^2\psi_1'(0)^2.
	\end{equation*}
	On using $\psi_1{''}(0)=\mathbb{V}\mathrm{ar}Y_1(1)$, we get
	\begin{align*}
		\mathbb{V}\mathrm{ar}Z_1(t_1)&=\mathbb{E}L_1(t_1)\mathbb{V}\mathrm{ar}Y_1(1)+\mathbb{E}L_1(t_1)^2(\mathbb{E}Y_1(1))^2-(\mathbb{E}L_1(t_1)\mathbb{E}Y_1(1))^2\\
		&=\mathbb{E}L_1(t_1)\mathbb{V}\mathrm{ar}Y_1(1)+\mathbb{V}\mathrm{ar}L_1(t_1)(\mathbb{E}Y_1(1))^2+(\mathbb{E}L_1(t_1)\mathbb{E}Y_1(1))^2-(\mathbb{E}L_1(t_1)\mathbb{E}Y_1(1))^2.
	\end{align*}
	Similarly, we have $\mathbb{V}\mathrm{ar}Z_2(t_2)=\mathbb{E}L_2(t_2)\mathbb{V}\mathrm{ar}Y_2(1)+\mathbb{V}\mathrm{ar}L_2(t_2)(\mathbb{E}Y_2(1))^2$, $t_2\ge0$. This completes the proof.
\end{proof}
\begin{remark}
	From Proposition 3.4 of \cite{Beghin2020}, it follows that all the mixed moments of bivariate inverse subordinator exist. So, the variance-covariance matrix (\ref{cvm}) is well defined. Also, on using Eq. (3.13) of \cite{Beghin2020}, the Laplace transform of covariance of $(Z_1(t_1),Z_2(t_2))$ is given by
	{\small\begin{equation*}
			\int_{0}^{\infty}\int_{0}^{\infty}e^{-z_1t_1-z_2t_2}\mathbb{C}\mathrm{ov}(Z_1(t_1),Z_2(t_2))\,\mathrm{d}t_1\,\mathrm{d}t_2=\frac{\mathbb{E}Y_1(1)Y_2(1)(B_1(z_1)+B_2(z_2)-B(z_1,z_2))}{z_1z_2B_1(z_1)B_2(z_2)B(z_1,z_2)},\ z_1>0,\,z_2>0,
	\end{equation*}}
	where $B_1$, $B_2$ and $B$  are defined in (\ref{mbif}) and (\ref{bibernsfncn}).
\end{remark}

Next, we obtain the auto correlation matrix of the process (\ref{tclp}).
\begin{theorem}\label{acvthm}
	For $(s_1,s_2)\in\mathbb{R}^2_+$ and $(t_1,t_2)\in\mathbb{R}^2_+$, the auto covariance matrix of (\ref{tclp}) is given by $(\mathbb{C}\mathrm{ov}(Z_i(s_i),Z_j(t_j)))_{1\leq i,j\leq2}$,
	where 
	\begin{equation}\label{acvf}
		\mathbb{C}\mathrm{ov}(Z_i(s_i),Z_j(t_j))=\begin{cases}
			\mathbb{C}\mathrm{ov}(L_j(s_j),L_j(t_j))(\mathbb{E}Y_j(1))^2+\mathbb{E}L_j(\min\{s_j, t_j\})\mathbb{V}\mathrm{ar}Y_j(1),\ i=j,\vspace{0.2cm}\\
			\mathbb{E}(Y_1(1)Y_2(1))\mathbb{C}\mathrm{ov}(L_i(s_i),L_j(t_j)),\ i\neq j.
		\end{cases}
	\end{equation}
\end{theorem}
\begin{proof}
	On using the independence of $Y_1(t_1)$ and $Y_2(t_2)$, we have
	\begin{equation*}
		\mathbb{E}(Y_1(s_1)Y_2(t_2))=s_1t_2\mathbb{E}(Y_1(1)Y_2(1))\ \text{and}\ 
		\mathbb{E}(Y_2(s_2)Y_1(t_1))=s_2t_1\mathbb{E}(Y_1(1)Y_2(1)).
	\end{equation*}
	If $s_1\leq t_1$ then on using the independent increments property of $Y_1(t_1)$, we have
	\begin{align*}
		\mathbb{E}(Y_1(s_1)Y_1(t_1))&=\mathbb{E}(Y_1(s_1)(Y_1(t_1)-Y_1(s_1)))+\mathbb{E}Y_1(s_1)^2\\
		&=\mathbb{E}Y_1(t_1-s_1)\mathbb{E}Y_1(s_1)+\mathbb{V}\mathrm{ar}(Y_1(s_1))+(\mathbb{E}Y_1(s_1))^2\\
		&=s_1(t_1-s_1)(\mathbb{E}Y_1(1))^2+\mathbb{V}\mathrm{ar}(Y_1(s_1))+s_1^2(\mathbb{E}Y_1(1))^2\\
		&=s_1t_1(\mathbb{E}Y_1(1))^2+s_1\mathbb{V}\mathrm{ar}Y_1(1),
	\end{align*}
	and similarly for $s_2\leq t_2$, we have
	\begin{equation*}
		\mathbb{E}(Y_2(s_2)Y_2(t_2))=s_2t_2(\mathbb{E}Y_2(1))^2+s_2\mathbb{V}\mathrm{ar}(Y_2(1)).
	\end{equation*}
	So,
	\begin{align*}
		\mathbb{E}(Z_1(s_1)Z_2(t_2))&=\int_{0}^{\infty}\int_{0}^{\infty}\mathbb{E}(Y_1(x_1)Y_2(x_2))\mathrm{Pr}(L_1(s_1)\in\mathrm{d}x_1,L_2(t_2)\in\mathrm{d}x_2)\\
		&=\mathbb{E}(Y_1(1)Y_2(1))\int_{0}^{\infty}\int_{0}^{\infty}x_1x_2\mathrm{Pr}(L_1(s_1)\in\mathrm{d}x_1,L_2(t_2)\in\mathrm{d}x_2)\\
		&=\mathbb{E}(Y_1(1)Y_2(1))\mathbb{E}(L_1(s_1)L_2(t_2)),
	\end{align*}
	and similarly, we have
	\begin{equation*}
		\mathbb{E}(Z_2(s_1)Z_1(t_1))=\mathbb{E}(Y_1(1)Y_2(1))\mathbb{E}(L_2(s_2)L_1(t_1)).
	\end{equation*}
	Moreover, from (\ref{tclp}), we get
	\begin{equation*}
		\mathbb{E}(Z_1(s_1)Z_1(t_1))=\mathbb{E}L_1(s_1)L_1(t_1)(\mathbb{E}Y_1(1))^2+\mathbb{E}L_1(s_1)\mathbb{V}\mathrm{ar}Y_1(1)
	\end{equation*}
	and
	\begin{equation*}
		\mathbb{E}(Z_2(s_2)Z_2(t_2))=\mathbb{E}L_2(s_2)L_2(t_2)(\mathbb{E}Y_2(1))^2+\mathbb{E}L_2(s_2)\mathbb{V}\mathrm{ar}Y_2(1).
	\end{equation*}
	Thus, by using (\ref{tclpmean}), we get 
	\begin{equation*}
		\mathbb{C}\mathrm{ov}(Z_i(s_i),Z_j(t_j))=\mathbb{E}(Z_i(s_i)Z_j(t_j))-\mathbb{E}L_i(s_i)\mathbb{E}L_j(t_j)\mathbb{E}Y_i(1)\mathbb{E}Y_j(1),\ 1\leq i,j\leq2.
	\end{equation*}
	This completes the proof.
\end{proof}

\begin{remark}
	If $\mathbb{E}Y_1(1)=\mathbb{E}Y_2(1)=0$ then from (\ref{acvf}), we have
	\begin{equation*}
		\mathbb{C}\mathrm{ov}(Z_i(s_i),Z_j(t_j))=\begin{cases}
			\mathbb{E}L_j(\min\{s_j, t_j\})\mathbb{V}\mathrm{ar}Y_j(1),\ i=j,\vspace{0.2cm}\\
			0,\ i\neq j,
		\end{cases}
	\end{equation*}
	that is, $\{Z_1(t_1), t_1\ge0\}$ and $\{Z_2(t_2), t_2\ge0\}$ are uncorrelated whenever $t_1\ne t_2$.
	
	Moreover, for $s_j=t_j$, $j=1,2$, (\ref{acvf}) reduces to
	\begin{equation*}
		\mathbb{C}\mathrm{ov}(Z_i(t_i),Z_j(t_j))=\begin{cases}
			\mathbb{V}\mathrm{ar}(L_j(t_j))(\mathbb{E}Y_j(1))^2+\mathbb{E}L_j(t_j)\mathbb{V}\mathrm{ar}Y_j(1),\ i=j,\vspace{0.2cm}\\
			\mathbb{E}(Y_1(1)Y_2(1))\mathbb{C}\mathrm{ov}(L_i(t_i),L_j(t_j)),\ i\neq j,
		\end{cases}
	\end{equation*}
	which agrees with (\ref{cvm}). 
\end{remark}
\begin{remark}\label{rein}
	Note that Proposition \ref{proptclp} and Theorem  \ref{acvthm} hold true if we replace the bivariate inverse subordinator $\{(L_1(t_1),L_2(t_2)),\ (t_1,t_2)\in\mathbb{R}^2_+\}$ in (\ref{tclp}) with any bivariate random process whose marginals, not necessarily independent, are non-negative and non-decreasing. For example, we can use the time-changed process (\ref{tpcomp}) in place of bivariate inverse subordinator.
\end{remark}
\begin{example}(\textbf{Bivariate fractional Poisson process})
	Let us consider a case where the processes $\{Y_1(t_1),\ t_1\ge0\}$ and $\{Y_2(t_2),\ t_2\ge0\}$ are Poisson processes with intensities $\lambda_1>0$ and $\lambda_2>0$, respectively. Also, the time changing process is a bivariate inverse $\alpha$-stable subordinator $\{(L_1^\alpha(t_1),L_2^\alpha(t_2)),\ (t_1,t_2)\in\mathbb{R}^2_+\}$, $0<\alpha<1$ associated with the bivariate stable subordinator of index $\alpha$ as given in Example 2.2 of \cite{Beghin2020}. For more details on multivariate stable processes, we refer the reader to \cite{D’Ovidio2014} and \cite{Meerschaert1999}, and references therein.
	
	The bivariate time-changed process $\{(Y_1(L_1^\alpha(t_1)),Y_2(L_2^\alpha(t_2))),\ (t_1,t_2)\in\mathbb{R}^2_+\}$ is an extension of the one parameter fractional Poisson process. Its joint probability mass function $\tilde{p}(n_1,n_2,t_1,t_2)=\mathrm{Pr}\{Y_1(L_1^\alpha(t_1))=n_1,Y_2(L_2^\alpha(t_2))=n_2\}$, $n_1\ge0$, $n_2\ge0$ solves the following system of differential equations (see \cite{Beghin2020}):
	\begin{equation*}
		\mathcal{D}_{t_1,t_2}^\alpha\tilde{p}(n_1,n_2,t_1,t_2)=-(\lambda_1+\lambda_2)\tilde{p}(n_1,n_2,t_1,t_2)+\lambda_1\tilde{p}(n_1-1,n_2,t_1,t_2) +\lambda_2\tilde{p}(n_1,n_2-1,t_1,t_2),
	\end{equation*}
	with $\tilde{p}(n_1,n_2,0,0)=\delta_0(n_1)\delta_0(n_2)$. Here, the operator $\mathcal{D}_{t_1,t_2}^\alpha$ is  defined as
	\begin{equation*}
		\mathcal{D}_{t_1,t_2}^\alpha f(t_1,t_2)=C\int_{0}^{\pi/2}(\cos\theta\partial_{t_1}+\sin\theta\partial_{t_2})^\alpha f(t_1,t_2)M(\mathrm{d}\theta)
	\end{equation*}
	 where $M(\mathrm{d}\theta)$ is a measure defined on the arc of unit circle and $C>0$ is a suitable constant.  
	
	From (\ref{acvf}), we have
	\begin{equation*}
		\mathbb{C}\mathrm{ov}(Y_i(L_i^\alpha(s_i)),Y_j(L_j^\alpha(t_j)))=\begin{cases}
			\lambda_j^2\mathbb{C}\mathrm{ov}(L_j^\alpha(s_j),L_j^\alpha(t_j))+\lambda_j\mathbb{E}L_j^\alpha(\min\{s_j, t_j\}),\ i=j,\vspace{0.2cm}\\
			\lambda_1\lambda_2\mathbb{C}\mathrm{ov}(L_i^\alpha(s_i),L_j^\alpha(t_j)),\ i\neq j.
		\end{cases}
	\end{equation*} 
\end{example}                                                  \subsection{Time-changed two parameter L\'evy process}       Here, we study the time-changed two parameter L\'evy process with rectangular increments. 

A collection of random variables $\{Y(t_1,t_2),\ (t_1,t_2)\in\mathbb{R}^2_+\}$ is called two parameter random process.
For $(s_1,s_2)\prec(t_1,t_2)$, the rectangular increments of two parameter random process on rectangle $R=(s_1,t_1]\times(s_2,t_2]$ is defined as follows:
\begin{equation*}
Y((s_1,t_1]\times(s_2,t_2])=Y(t_1,t_2)-Y(s_1,t_2)-Y(t_1,s_2)+Y(s_1,s_2).
\end{equation*}

We consider the following definition of stationary and independent increments of two parameter processes:
\begin{definition}\label{instain}
Let $\{Y(t_1,t_2),\ (t_1,t_2)\in\mathbb{R}^2_+\}$ be a two parameter random process. Then,\\
\noindent (i) it has stationary increments if for $(s_1,s_2)\prec(t_1,t_2)$, the distribution of $Y((s_1,t_1]\times(s_2,t_2])$ depends only on the area of rectangle $R=(s_1,t_1]\times(s_2,t_2]$, that is, $(s_2-s_1)(t_2-t_1)$,\\
\noindent (ii) it has independent increments if $R_1,R_2,\dots,R_m$ are disjoint rectangles of type $R$ then the random variables $Y(R_1),Y(R_2),\dots,Y(R_m)$ are independent of each other.
\end{definition}
Now, we recall the definition of two parameter L\'evy process. The existence and characterizations of such processes is studied in  \cite{Starf1972}. Such extension of L\'evy processes with multidimensional parameter is studied by various authors. For some L\'evy processes with rectangular increments, we refer the reader to \cite{Adler1983}, \cite{Dalang1991} and \cite{Khoshnevisan1999}.
\begin{definition}(\textbf{Two-parameter L\'evy process})
A two parameter $\mathbb{R}^d$-valued random process $\{Y(t_1,t_2)$, $ (t_1,t_2)\in\mathbb{R}^2_+\}$ is called two parameter L\'evy process if \\ 
\noindent (i) $Y(0,t_2)=Y(t_1,0)=0$ almost surely,\\
\noindent (ii) it has stationary and independent rectangular increments,\\
\noindent (iii) it is continuous in probability.
\end{definition}  

The characteristic function of two parameter L\'evy process is given by $\mathbb{E}e^{i\langle\xi, Y(t_1,t_2)\rangle}=e^{-t_1t_2\Psi(\xi)}$, $\xi\in\mathbb{R}^d_+$ where $\langle\cdot,\cdot\rangle$ denotes the standard inner product on Euclidean space, and $\Psi$ is the L\'evy exponent of $Y(t_1,t_2)$ defined as follows:
\begin{equation*}
	\Psi(\xi)=i\langle b,\xi\rangle+\frac{\langle\Sigma\xi,\xi\rangle}{2}-\int_{\mathbb{R}^d-\{0\}}\Big(e^{i\langle\xi,x\rangle}-1-i\langle\xi,x\rangle\textbf{1}_{||x||\leq 1}\Big)\nu(\mathrm{d}x)
\end{equation*}                                            
for some $b\in\mathbb{R}^d$ and $\Sigma$ is a $d\times d$ symmetric  positive semidefinite matrix, and $\nu$ is  the L\'evy measure on $\mathbb{R}^d$.  
Also, in \cite{Starf1972}, it is shown that the characteristic function of $Y(t_1,t_2)$ can be written as $\mathbb{E}e^{i\xi Y(t_1,t_2)}=\phi^{t_1t_2}(\xi)$ for all $(t_1,t_2)\in\mathbb{R}^2_+$, where $\phi(\cdot)$ is the characteristic function of an infinitely divisible distribution on $\mathbb{R}^d$.    

Next, we study a time-changed two parameter L\'evy process. Let $\{Y(t_1,t_2),\ (t_1,t_2)\in\mathbb{R}^2_+\}$ be a real two parameter L\'evy process and $\{(L_1(t_1),L_2(t_2)),\ (t_1,t_2)\in\mathbb{R}^2_+\}$ be the bivariate inverse subordinator which is independent of it. We consider the following time-changed two parameter process:
\begin{equation}\label{tptclp}
	\tilde{Y}(t_1,t_2)\coloneqq Y(L_1(t_1),L_2(t_2)),\ (t_1,t_2)\in\mathbb{R}^2_+.
\end{equation}
Its characteristic function is given by
\begin{align}
	\mathbb{E}e^{i\xi\tilde{Y}(t_1,t_2)}&=\int_{0}^{\infty}\int_{0}^{\infty}\mathbb{E}e^{i\xi Y(x_1,x_2)}\mathrm{Pr}\{L_1(t_1)\in\mathrm{d}x_1,L_2(t_2)\in\mathrm{d}x_2\}\nonumber\\
	&=\int_{0}^{\infty}\int_{0}^{\infty}e^{-x_1x_2\Psi(\xi)}\mathrm{Pr}\{L_1(t_1)\in\mathrm{d}x_1,L_2(t_2)\in\mathrm{d}x_2\}\nonumber\\
	&=\mathbb{E}\exp(-L_1(t_1)L_2(t_2)\Psi(\xi)),\ \xi\in\mathbb{R}.\label{tptclap}
\end{align}

Next, we derive the explicit expressions for its mean, variance and auto covariance. As mentioned in Remark \ref{rein} for the case of bivariate L\'evy processes, all the results obtained here hold true if we replace the bivariate inverse subordinator in (\ref{tptclp}) with any bivariate process with non-negative and non-decreasing marginals. 
\begin{theorem}\label{thmtptc}
	Let $\{\tilde{Y}(t_1,t_2),\ (t_1,t_2)\in\mathbb{R}^2_+\}$ be the time-changed process defined in (\ref{tptclp}). For all $(t_1,t_2)\in\mathbb{R}^2_+$, \\
	\noindent(i) if first moments of $Y(1,1)$ and $L_1(t_1)L_2(t_2)$ exist then the first moment of (\ref{tptclp}) exists and it is given by
	\begin{equation*}
		\mathbb{E}\tilde{Y}(t_1,t_2)=\mathbb{E}L_1(t_1)L_2(t_2)\mathbb{E}Y(1,1),
	\end{equation*}
	\noindent (ii) if the second moments of $Y(1,1)$ and $L_1(t_1)L_2(t_2)$ exist then
	\begin{equation*}
		\mathbb{V}\mathrm{ar}\tilde{Y}(t_1,t_2)=\mathbb{E}L_1(t_1)L_2(t_2)\mathbb{V}\mathrm{ar}Y(1,1)+\mathbb{V}\mathrm{ar}L_1(t_1)L_2(t_2)(\mathbb{E}Y(1,1))^2,
	\end{equation*}
	\noindent (iii) for $(s_1,s_2)\prec(t_1,t_2)$, we have
	\begin{equation*}
		\mathbb{C}\mathrm{ov}(\tilde{Y}(s_1,s_2),\tilde{Y}(t_1,t_2))=\mathbb{E}L_1(s_1)L_2(s_2)\mathbb{V}\mathrm{ar}Y(1,1)+\mathbb{C}\mathrm{ov}(L_1(s_1)L_2(s_2),L_1(t_1)L_2(t_2))(\mathbb{E}Y(1,1))^2.
	\end{equation*}
\end{theorem}
\begin{proof}
	Note that $\mathbb{E}Y(1,1)=i\Psi'(0)$ and $\mathbb{V}\mathrm{ar}Y(1,1)=\Psi''(0)$. So, \\
	\noindent (i) from (\ref{tptclap}), we have
	\begin{equation}\label{tptcmean}
		\mathbb{E}\tilde{Y}(t_1,t_2)=-i\partial_\xi\mathbb{E}e^{i\xi\tilde{Y}(t_1,t_2)}|_{\xi=0}=i\mathbb{E}L_1(t_1)L_2(t_2)\Psi'(0)=\mathbb{E}L_1(t_1)L_2(t_2)\mathbb{E}Y(1,1).
	\end{equation}
	\noindent (ii) Also, we have
	\begin{align*}
		\mathbb{E}\tilde{Y}(t_1,t_2)^2&=-\partial_\xi^2\mathbb{E}e^{i\xi\tilde{Y}(t_1,t_2)}|_{\xi=0}\\
		&=\mathbb{E}L_1(t_1)L_2(t_2)\Psi''(0)-\mathbb{E}(L_1(t_1)L_2(t_2))^2(\Psi'(0))^2\\
		&=\mathbb{E}L_1(t_1)L_2(t_2)\mathbb{V}\mathrm{ar}Y(1,1)+\mathbb{E}(L_1(t_1)L_2(t_2))^2(\mathbb{E}Y(1,1))^2\\
		&=\mathbb{E}L_1(t_1)L_2(t_2)\mathbb{V}\mathrm{ar}Y(1,1)+\mathbb{V}\mathrm{ar}L_1(t_1)L_2(t_2)(\mathbb{E}Y(1,1))^2+(\mathbb{E}L_1(t_1)L_2(t_2)\mathbb{E}Y(1,1))^2.
	\end{align*}
	Thus, by using (\ref{tptcmean}), we get the required variance of (\ref{tptclp}).\\
	\noindent (iii) For $(s_1,s_2)\prec(t_1,t_2)$ in $\mathbb{R}^2_+$ and using the independent rectangular increments property of the outer process $Y(t_1,t_2)$, we have
	\begin{align*}
		\mathbb{E}Y(s_1,s_2)Y(t_1,t_2)&=\mathbb{E}Y(s_1,s_2)(Y(t_1,t_2)-Y(s_1,t_2)-Y(t_1,s_2)+Y(s_1,s_2))\\
		&\ \ +\mathbb{E}Y(s_1,s_2)Y(s_1,t_2)+\mathbb{E}Y(s_1,s_2)Y(t_1,s_2)-\mathbb{E}Y(s_1,s_2)^2\\
		&=\mathbb{E}Y(s_1,s_2)\mathbb{E}(Y(t_1,t_2)-Y(s_1,t_2)-Y(t_1,s_2)+Y(s_1,s_2))\\
		&\ \ +\mathbb{E}Y(s_1,s_2)\mathbb{E}(Y(s_1,t_2)-Y(s_1,s_2))+\mathbb{E}Y(s_1,s_2)^2\\
		&\ \ +\mathbb{E}Y(s_1,s_2)\mathbb{E}(Y(t_1,s_2)-Y(s_1,s_2))\\
		&=s_1s_2(t_2-s_2)(t_1-s_1)(\mathbb{E}Y(1,1))^2+s_1^2s_2(t_2-s_2)(\mathbb{E}Y(1,1))^2\\
		&\ \ +\mathbb{V}\mathrm{ar}Y(s_1,s_2)+s_1^2s_2^2(\mathbb{E}Y(1,1))^2+s_1s_2^2(t_1-s_1)(\mathbb{E}Y(1,1))^2\\
		&=s_1s_2t_1t_2(\mathbb{E}Y(1,1))^2+s_1s_2\mathbb{V}\mathrm{ar}Y(1,1).
	\end{align*}
	Thus, from (\ref{tptclp}), it follows that
	\begin{equation*}
		\mathbb{E}\tilde{Y}(s_1,s_2)\tilde{Y}(t_1,t_2)=\mathbb{E}L_1(s_1)L_2(s_2)L_1(t_1)L_2(t_2)(\mathbb{E}Y(1,1))^2+\mathbb{E}L_1(s_1)L_2(s_2)\mathbb{V}\mathrm{ar}Y(1,1).
	\end{equation*}
	This completes the proof.
\end{proof}
\begin{remark}
	If $\mathbb{E}Y(1,1)=0$ then $\tilde{Y}(t_1,t_2)$ has zero mean. Also, we have $\mathbb{V}\mathrm{ar}\tilde{Y}(t_1,t_2)=\mathbb{E}L_1(t_1)L_2(t_2)\mathbb{V}\mathrm{ar}Y(1,1)$ and
	$		\mathbb{C}\mathrm{ov}(\tilde{Y}(s_1,s_2),\tilde{Y}(t_1,t_2))=\mathbb{E}L_1(s_1)L_2(s_2)\mathbb{V}\mathrm{ar}Y(1,1).
	$
	Thus, the auto correlation function of the two parameter process (\ref{tptclp}) is given by
	\begin{equation*}
		\mathbb{C}\mathrm{orr}(\tilde{Y}(s_1,s_2),\tilde{Y}(t_1,t_2))=\sqrt{\frac{\mathbb{E}L_1(s_1)L_2(s_2)}{\mathbb{E}L_1(t_1)L_2(t_2)}},\ (s_1,s_2)\prec(t_1,t_2).
	\end{equation*}
\end{remark}
\begin{example}
	Let $\{L_{\beta_1}(t_1),\ t_1\ge0\}$ and $\{L_{\beta_2}(t_2),\ t_2\ge0\}$ be inverse stable subordinators with indices $\beta_1\in(0,1)$ and $\beta_2\in(0,1)$, respectively, and let $\{\mathscr{N}(t_1,t_2),\ (t_1,t_2)\in\mathbb{R}^2_+\}$ be the PRF with parameter $\lambda>0$. It is assumed that all these processes are independent of each other.  Leonenko and Merzbach \cite{Leonenko2015} introduced and studied the following time-changed PRF:
	$
		\{\mathscr{N}(L_{\beta_1}(t_1),L_{\beta_2}(t_2)),\ (t_1,t_2)\in\mathbb{R}^2_+\}.
	$ For more details on this process, we refer the reader to \cite{Aletti2018} and \cite{Kataria2024}.
	
	 For $i=1,2$, it is known that (see \cite{Leonenko2014})
	\begin{equation*}
		\mathbb{E}L_{\beta_i}(t_i)=\frac{t_i^{\beta_i}}{\Gamma(\beta_i+1)},\  \mathbb{V}\mathrm{ar}L_{\beta_i}(t_i)=t_i^{2\beta_i}\bigg(\frac{2}{\Gamma(2\beta_i+1)}-\frac{1}{\Gamma(\beta_i+1)^2}\bigg)
	\end{equation*}
	and for $s_i<t_i$, we have
	\begin{equation*}
		\mathbb{C}\mathrm{ov}(L_{\beta_i}(s_i),L_{\beta_i}(t_i))=\frac{1}{\Gamma(\beta_i+1)\Gamma(\beta_i)}\int_{0}^{s_i}((t_i-x)^{\beta_i}+(s_i-x)^{\beta_i})x^{\beta_i-1}\,\mathrm{d}x-\frac{(s_it_i)^{\beta_i}}{\Gamma(\beta_i+1)^2}.
	\end{equation*}
	So, by using the independence of $L_{\beta_i}(t_i)$'s, we get
	\begin{align*}
		\mathbb{V}\mathrm{ar}L_{\beta_1}(t_1)L_{\beta_2}(t_2)&=\mathbb{E}(L_{\beta_1}(t_1)L_{\beta_2}(t_2))^2-(\mathbb{E}L_{\beta_1}(t_1)L_{\beta_2}(t_2))^2\\
		&=\mathbb{V}\mathrm{ar}L_{\beta_1}(t_1)\mathbb{E}L_{\beta_2}(t_2)^2+(\mathbb{E}L_{\beta_1}(t_1))^2\mathbb{V}\mathrm{ar}L_{\beta_2}(t_2)\\
		&=\mathbb{V}\mathrm{ar}L_{\beta_1}(t_1)\mathbb{V}\mathrm{ar}L_{\beta_2}(t_2)+\mathbb{V}\mathrm{ar}L_{\beta_1}(t_1)(\mathbb{E}L_{\beta_2}(t_2))^2+(\mathbb{E}L_{\beta_1}(t_1))^2\mathbb{V}\mathrm{ar}L_{\beta_2}(t_2)\\
		&=t_1^{2\beta_1}t_2^{2\beta_2}\bigg(\frac{2}{\Gamma(2\beta_1+1)}-\frac{1}{\Gamma(\beta_1+1)^2}\bigg)\bigg(\frac{2}{\Gamma(2\beta_2+1)}-\frac{1}{\Gamma(\beta_2+1)^2}\bigg)\\
		&\ \ +\bigg(\frac{2}{\Gamma(2\beta_1+1)}-\frac{1}{\Gamma(\beta_1+1)^2}\bigg)\frac{t_1^{2\beta_1}t_2^{2\beta_2}}{\Gamma(\beta_2+1)^2}\\
		&\ \ +\bigg(\frac{2}{\Gamma(2\beta_2+1)}-\frac{1}{\Gamma(\beta_2+1)^2}\bigg)\frac{t_1^{2\beta_1}t_2^{2\beta_2}}{\Gamma(\beta_1+1)^2}\\
		&=t_1^{2\beta_1}t_2^{2\beta_2}\bigg(\frac{4}{\Gamma(2\beta_1+1)\Gamma(2\beta_2+1)}-\frac{1}{\Gamma(\beta_1+1)^2\Gamma(\beta_2+1)^2}\bigg).
	\end{align*}
	Also, for $(s_1,s_2)\prec(t_1,t_2)$, we have
	\begin{align*}
		\mathbb{C}\mathrm{ov}(&L_{\beta_1}(s_1)L_{\beta_2}(s_2),L_{\beta_1}(t_1)L_{\beta_2}(t_2))\\
		&=\mathbb{E}L_{\beta_1}(s_1)L_{\beta_2}(s_2)L_{\beta_1}(t_1)L_{\beta_2}(t_2)-\mathbb{E}L_{\beta_1}(s_1)L_{\beta_2}(s_2)\mathbb{E}L_{\beta_1}(t_1)L_{\beta_2}(t_2)\\
		&=\mathbb{E}L_{\beta_1}(s_1)L_{\beta_1}(t_1)\mathbb{E}L_{\beta_2}(s_2)L_{\beta_2}(t_2)-\mathbb{E}L_{\beta_1}(s_1)\mathbb{E}L_{\beta_1}(t_1)\mathbb{E}L_{\beta_2}(s_2)\mathbb{E}L_{\beta_2}(t_2)\\
		&=\mathbb{C}\mathrm{ov}(L_{\beta_1}(s_1),L_{\beta_1}(t_1))\mathbb{E}L_{\beta_2}(s_2)L_{\beta_2}(t_2)+\mathbb{E}L_{\beta_1}(s_1)\mathbb{E}L_{\beta_1}(t_1)\mathbb{C}\mathrm{ov}(L_{\beta_2}(s_2),L_{\beta_2}(t_2))\\
		&=\mathbb{C}\mathrm{ov}(L_{\beta_1}(s_1),L_{\beta_1}(t_1))\mathbb{C}\mathrm{ov}(L_{\beta_2}(s_2),L_{\beta_2}(t_2))+\mathbb{C}\mathrm{ov}(L_{\beta_1}(s_1),L_{\beta_1}(t_1))\mathbb{E}L_{\beta_2}(s_2)\mathbb{E}L_{\beta_2}(t_2)\\
		&\ \ +\mathbb{E}L_{\beta_1}(s_1)\mathbb{E}L_{\beta_1}(t_1)\mathbb{C}\mathrm{ov}(L_{\beta_2}(s_2),L_{\beta_2}(t_2))\\
		&=\bigg(\frac{1}{\Gamma(\beta_1+1)\Gamma(\beta_1)}\int_{0}^{s_1}((t_1-x)^{\beta_1}+(s_1-x)^{\beta_1})x^{\beta_1-1}\,\mathrm{d}x-\frac{(s_1t_1)^{\beta_1}}{\Gamma(\beta_1+1)^2}\bigg)\\
		&\ \ \cdot\bigg(\frac{1}{\Gamma(\beta_2+1)\Gamma(\beta_2)}\int_{0}^{s_2}((t_2-x)^{\beta_2}+(s_2-x)^{\beta_2})x^{\beta_2-1}\,\mathrm{d}x-\frac{(s_2t_2)^{\beta_2}}{\Gamma(\beta_2+1)^2}\bigg)\\
		&\ \ +\bigg(\frac{1}{\Gamma(\beta_1+1)\Gamma(\beta_1)}\int_{0}^{s_1}((t_1-x)^{\beta_1}+(s_1-x)^{\beta_1})x^{\beta_1-1}\,\mathrm{d}x-\frac{(s_1t_1)^{\beta_1}}{\Gamma(\beta_1+1)^2}\bigg)\frac{(s_2t_2)^{\beta_2}}{\Gamma(\beta_2+1)^2}\\
		&\ \ + \bigg(\frac{1}{\Gamma(\beta_2+1)\Gamma(\beta_2)}\int_{0}^{s_2}((t_2-x)^{\beta_2}+(s_2-x)^{\beta_2})x^{\beta_2-1}\,\mathrm{d}x-\frac{(s_2t_2)^{\beta_2}}{\Gamma(\beta_2+1)^2}\bigg)\frac{(s_1t_1)^{\beta_1}}{\Gamma(\beta_1+1)^2}\\
		&=\prod_{i=1}^{2}\frac{1}{\Gamma(\beta_i+1)\Gamma(\beta_i)}\int_{0}^{s_i}((t_i-x_i)^{\beta_i}+(s_i-x_i)^{\beta_i})x_i^{\beta_i-1}\,\mathrm{d}x_i-\prod_{i=1}^{2}\frac{(s_it_i)^{\beta_i}}{\Gamma(\beta_i+1)^2}.
	\end{align*}
	Thus, from Theorem \ref{thmtptc} it follows that
	\begin{align*}
		\mathbb{E}\mathscr{N}(L_{\beta_1}(t_1),L_{\beta_2}(t_2))&=\frac{\lambda t_1^{\beta_1}t_2^{\beta_2}}{\Gamma(\beta_1+1)\Gamma(\beta_2+1)},\\
		\mathbb{V}\mathrm{ar}\mathscr{N}(L_{\beta_1}(t_1),L_{\beta_2}(t_2))&=\frac{\lambda t_1^{\beta_1}t_2^{\beta_2}}{\Gamma(\beta_1+1)\Gamma(\beta_2+1)}\\
		&\ \ +(\lambda t_1^{\beta_1}t_2^{\beta_2})^2\bigg(\frac{4}{\Gamma(2\beta_1+1)\Gamma(2\beta_2+1)}-\frac{1}{\Gamma(\beta_1+1)^2\Gamma(\beta_2+1)^2}\bigg)
	\end{align*}
	and
	\begin{align*}
		\mathbb{C}\mathrm{ov}(\mathscr{N}(L_{\beta_1}(s_1)&,L_{\beta_2}(s_2)),\mathscr{N}(L_{\beta_1}(t_1),L_{\beta_2}(t_2)))\\
		&=\prod_{i=1}^{2}\frac{\lambda}{\Gamma(\beta_i+1)\Gamma(\beta_i)}\int_{0}^{s_i}((t_i-x_i)^{\beta_i}+(s_i-x_i)^{\beta_i})x_i^{\beta_i-1}\,\mathrm{d}x_i-\prod_{i=1}^{2}\frac{\lambda(s_it_i)^{\beta_i}}{\Gamma(\beta_i+1)^2}\\
		&\ \ +\frac{\lambda s_1^{\beta_1}s_2^{\beta_2}}{\Gamma(\beta_1+1)\Gamma(\beta_2+1)},\ (s_1,s_2)\prec(t_1,t_2),
	\end{align*}
	which agree with the results obtained in \cite{Leonenko2015}.
\end{example}
\section{Some time-changed Poisson random fields}\label{sec3}
In this section, we study some time-changed variants of the PRF where the time changing components used are stable subordinators, inverse stable subordinators and their composition.

 Let $\{\mathscr{N}(t_1,t_2),\ (t_1,t_2)\in\mathbb{R}^2_+\}$ be the PRF and $\{H^{\alpha_i,\beta_i}(t),\ t\ge0\}$,  $\alpha_i\in(0,1]$, $\beta_i\in(0,1]$, $i=1,2$ be two random processes as defined in (\ref{driftdef}). Let us assume that all three processes are independent of each other. We consider a time-changed random field $\{\mathscr{N}_{\alpha_1,\beta_1}^{\alpha_2,\beta_2}(t_1,t_2),\ (t_1,t_2)\in\mathbb{R}^2_+\}$ defined as follows:
\begin{equation}\label{tprf}
\mathscr{N}_{\alpha_1,\beta_1}^{\alpha_2,\beta_2}(t_1,t_2)\coloneqq\mathscr{N}(H^{\alpha_1,\beta_1}(t_1),H^{\alpha_2,\beta_2}(t_2)),\ (t_1,t_2)\in\mathbb{R}^2_+.
\end{equation}
\begin{remark}
	For $\alpha_1=\alpha_2=1$, the process (\ref{tprf}) reduces to the following time-changed PRF introduced and studied in \cite{Leonenko2015}:
\begin{equation*}\label{fprf}
		\mathscr{N}_{1,\beta_1}^{1,\beta_2}(t_1,t_2)\coloneqq\mathscr{N}(L_{\beta_1}(t_1),L_{\beta_2}(t_2)),\ (t_1,t_2)\in\mathbb{R}^2_+,
\end{equation*}
	where $\{L_{\beta_1}(t_1),\ t_1\ge0\}$ and $\{L_{\beta_2}(t_2),\ t_2\ge0\}$ are independent inverse stable subordinators as defined in (\ref{insubdef}), and are independent of the PRF. Recently, Kataria and Vishwakarma \cite{Kataria2024} defined a fractional Poisson random field $\{\mathcal{N}_{\beta_1,\beta_2}(t_1,t_2),\ (t_1,t_2)\in\mathbb{R}^2_+\}$ whose distribution $p_{\beta_1,\beta_2}(n,t_1,t_2)=\mathrm{Pr}\{\mathcal{N}_{\beta_1,\beta_2}(t_1,t_2)=n\}$, $n\ge0$ solves the following system of fractional differential equations:
	\begin{equation*}
	\mathcal{D}_{t_2}^{\beta_2}\mathcal{D}_{t_1}^{\beta_1}p_{\beta_1,\beta_2}(n,t_1,t_2)=(n+1)\lambda p_{\beta_1,\beta_2}(n+1,t_1,t_2)-(2n+1)\lambda p_{\beta_1,\beta_2}(n,t_1,t_2)+n\lambda p_{\beta_1,\beta_2}(n-1,t_1,t_2),
	\end{equation*}
	with $p_{\beta_1,\beta_2}(0,0,0)=1$. Here, $\mathcal{D}_t^\beta$ is the Caputo fractional derivative operator as defined in (\ref{caputo}). Its solution is given by
\begin{equation*}
p_{\beta_1,\beta_2}(n,t_1,t_2)=\sum_{k=n}^{\infty}\frac{(-1)^{k-n}(k)_{k-n}(k)_{n}(\lambda t_1^{\beta_1}t_2^{\beta_2})^{k}}{\Gamma(k\beta_1+1)\Gamma(k\beta_2+1)},\ n\ge0,
\end{equation*}
where $(k)_n=k(k-1)\dots(k-n+1)$ and $(k)_0=1$.	Moreover, they established that $\mathcal{N}_{\beta_1,\beta_2}(t_1,t_2)\overset{d}{=}\mathscr{N}_{1,\beta_1}^{1,\beta_2}(t_1,t_2)$ for all $(t_1,t_2)\in\mathbb{R}^2_+$.
\end{remark}

For $i=1,2$, let $f_{\alpha_i,\beta_i}(x_i,t_i)$, $x_i\ge0$, $t_i\ge0$ be the density of $H^{\alpha_i,\beta_i}(t_i)$.  Then, the distribution $p_{\alpha_1,\beta_1}^{\alpha_2,\beta_2}(n,t_1,t_2)=\mathrm{Pr}\{\mathscr{N}_{\alpha_1,\beta_1}^{\alpha_2,\beta_2}(t_1,t_2)=n\}$, $n\ge0$  of (\ref{tprf}) is given by
\begin{align*}
p_{\alpha_1,\beta_1}^{\alpha_2,\beta_2}(n,t_1,t_2)&=\int_{0}^{\infty}\int_{0}^{\infty}p(n,s_1,s_2)f_{\alpha_1,\beta_1}(s_1,t_1)f_{\alpha_2,\beta_2}(s_2,t_2)\,\mathrm{d}s_1\,\mathrm{d}s_2\nonumber\\
		&=\frac{\lambda^n}{n!}\int_{0}^{\infty}\int_{0}^{\infty}(s_1s_2)^ne^{-\lambda s_1s_2}f_{\alpha_1,\beta_1}(s_1,t_1)f_{\alpha_2,\beta_2}(s_2,t_2)\,\mathrm{d}s_1\,\mathrm{d}s_2\nonumber\\
		&=\frac{(-\lambda\partial_\lambda)^n}{n!}\int_{0}^{\infty}\int_{0}^{\infty}e^{-\lambda s_1s_2}f_{\alpha_1,\beta_1}(s_1,t_1)f_{\alpha_2,\beta_2}(s_2,t_2)\,\mathrm{d}s_1\,\mathrm{d}s_2\nonumber\\
		&=\frac{(-\lambda\partial_\lambda)^n}{n!}\int_{0}^{\infty}E_{\beta_2,1}(-t_2^{\beta_2}(\lambda s_1)^{\alpha_2})f_{\alpha_1,\beta_1}(s_1,t_1)\,\mathrm{d}s_1,
	\end{align*}
	where we have used (\ref{driftlap}) in the last step.

\begin{remark}For $\beta_1=1$ and $\alpha_2=1$, the process (\ref{tprf}) reduces to 
$\mathscr{N}_{\alpha_1,1}^{1,\beta_2}(t_1,t_2)\coloneqq\mathscr{N}(S_{\alpha_1}(t_1),L_{\beta_2}(t_2))$, $ (t_1,t_2)\in\mathbb{R}^2_+$,
where $\{S_{\alpha_1}(t_1),\ t_1\ge0\}$ and $\{L_{\beta_2}(t_2),\ t_2\ge0\}$ are independent $\alpha_1$-stable subordinator and inverse $\beta_2$-stable subordinator. Its distribution is given by
	\begin{align*}
		p_{\alpha_1,1}^{1,\beta_2}(n,t_1,t_2)&=\int_{0}^{\infty}\int_{0}^{\infty}p(n,s_1,s_2)\mathrm{Pr}\{S_{\alpha_1}(t_1)\in\mathrm{d}s_1\}\mathrm{Pr}\{L_{\beta_2}(t_2)\in\mathrm{d}s_2\}\\
		&=\frac{(-\lambda\partial_\lambda)^n}{n!}\int_{0}^{\infty}\int_{0}^{\infty}e^{-\lambda s_1s_2}\mathrm{Pr}\{S_{\alpha_1}(t_1)\in\mathrm{d}s_1\}\mathrm{Pr}\{L_{\beta_2}(t_2)\in\mathrm{d}s_2\}\\
		&=\frac{(-\lambda\partial_\lambda)^n}{n!}\int_{0}^{\infty}e^{-t_1(\lambda s_2)^{\alpha_1}}\mathrm{Pr}\{L_{\beta_2}(t_2)\in\mathrm{d}s_2\},
	\end{align*}
	which on taking the Laplace transform with respect to $t_2$ and using (\ref{inslap}) yields
\begin{align}
	\int_{0}^{\infty}e^{-z_2t_2}\mathrm{Pr}\{\mathscr{N}(S_{\alpha_1}(t_1),L_{\beta_2}(t_2))=n\}\,\mathrm{d}t_2&=\frac{(-\lambda\partial_\lambda)^n}{n!}z_2^{\beta_2-1}\int_{0}^{\infty}e^{-t_1(\lambda s_2)^{\alpha_1}}e^{-z_2^{\beta_2}s_2}\,\mathrm{d}s_2\nonumber\\
	&=\frac{(-\lambda\partial_\lambda)^n}{n!}z_2^{\beta_2-1}\sum_{k=0}^{\infty}\frac{(-t_1\lambda^{\alpha_1})^k}{k!}\int_{0}^{\infty}s_2^{\alpha_1 k}e^{-z_2^{\beta_2}s_2}\,\mathrm{d}s_2\nonumber\\
	&=\frac{(-\lambda\partial_\lambda)^n}{n!}\sum_{k=0}^{\infty}\frac{(-t_1\lambda^{\alpha_1})^k}{k!}\frac{\Gamma(\alpha_1 k+1)}{z_2^{\alpha_1\beta_2k+1}}.\label{dc1}
\end{align}
On taking the inverse Laplace transform with respect to $z_2$ on both sides of (\ref{dc1}), we get
\begin{align*}
	\mathrm{Pr}\{\mathscr{N}(S_{\alpha_1}(t_1),L_{\beta_2}(t_2))=n\}&=\frac{(-\lambda\partial_\lambda)^n}{n!}\sum_{k=0}^{\infty}\frac{\Gamma(\alpha_1 k+1)(-t_1t_2^{\alpha_1\beta_2}\lambda^{\alpha_1})^k}{k!\Gamma(\alpha_1\beta_2k+1)}\nonumber\\
	&=\frac{(-\lambda\partial_\lambda)^n}{n!}{}_1\Psi_1\left[\begin{matrix}
		(1,\alpha_1)\\\\
		(1,\alpha_1\beta_2)
	\end{matrix}\Bigg|-t_1t_2^{\alpha_1\beta_2}\lambda^{\alpha_1}\right],\ n\ge0,
\end{align*}
where ${}_l\Psi_m$, $1\leq l,m<\infty$ is the generalized Wright function defined as follows (see \cite{Kilbas2006}, p. 56):
\begin{equation}\label{genwrit}
	{}_l\Psi_m\left[\begin{matrix}
		(a_1,\alpha_1),\dots,(a_l,\alpha_l)\\\\
		(b_1,\beta_1),\dots,(b_m,\alpha_m)
	\end{matrix}\Bigg| x \right]=\sum_{n=0}^{\infty}\frac{\prod_{i=1}^{l}\Gamma(a_i+n\alpha_i)x^n}{\prod_{j=1}^{m}\Gamma(b_j+n\beta_j)n!},\ x\in\mathbb{R}.
\end{equation}
\end{remark}

 Next, we obtain the closed form expression for the distribution of time-changed PRF $\{\mathscr{N}(H^{\alpha_1,\beta_1}$ $(t_1),t_2),\ (t_1,t_2)\in\mathbb{R}^2_+\}$ and derive its governing system of differential equations. In the case of Poisson process, Beghin and  D'Ovidio \cite{Beghin2014} used the Markov property of Poisson process to derive the corresponding governing system of  equations. However, we do not have any such property for the PRF. Here, we use the governing equation for the density of $H^{\alpha_1,\beta_1}(t_1)$ to prove the following result:
\begin{theorem}\label{thmtc1}
For $(t_1,t_2)\in\mathbb{R}^2_+$, the distribution $p_{\alpha_1}^{\beta_1}(n,t_1,t_2)=\mathrm{Pr}\{\mathscr{N}_{\alpha_1,\beta_1}^{1,1}(t_1,t_2)=n\}$, $n\ge0$  is given by
\begin{equation}\label{dist1f}
	p_{\alpha_1}^{\beta_1}(n,t_1,t_2)=\frac{(-\lambda\partial_\lambda)^n}{n!}E_{\beta_1,1}(-(\lambda t_2)^{\alpha_1}t_1^{\beta_1}),
\end{equation}
which solves the following system of fractional differential equations:
\begin{equation}\label{fprftc1}
	\mathcal{D}^{\beta_1}_{t_1}p_{\alpha_1}^{\beta_1}(n,t_1,t_2)=-\lambda^{\alpha_1}t_2^{\alpha_1}(I-B)^{\alpha_1}p_{\alpha_1}^{\beta_1}(n,t_1,t_2),\ n\ge0,\ t_2\ge0,
\end{equation}
with initial condition $p_{\alpha_1}^{\beta_1}(0,0,t_2)=1$. Here, $\mathcal{D}^{\beta_1}_{t_1}$ is the Caputo fractional derivative operator as defined in (\ref{caputo}), and 
\begin{equation}\label{grnbino}
	(I-B)^{\alpha_1}=\sum_{r=0}^{\infty}(-1)^r\frac{(\alpha_1)_r}{r!}B^r,
\end{equation}
where $(\alpha_1)_r=\alpha_1(\alpha_1-1)\dots(\alpha_1-r+1)$.
\end{theorem}
\begin{proof}
Let $f_{\alpha_1,\beta_1}(x_1,t_1)$ be the density of $H^{\alpha_1,\beta_1}(t_1)$, $x_1\ge0$, $t_1>0$. Then, for $\alpha_2=\beta_2=1$ and $n\ge0$, from (\ref{tprf}), we have
\begin{align}
	p_{\alpha_1}^{\beta_1}(n,t_1,t_2)&=\int_{0}^{\infty}p(n,s_1,t_2)f_{\alpha_1,\beta_1}(s_1,t_1)\,\mathrm{d}s_1\label{apf1}\\
	&=\frac{(\lambda t_2)^n}{n!}\int_{0}^{\infty}s_1^ne^{-\lambda s_1t_2}f_{\alpha_1,\beta_1}(s_1,t_1)\,\mathrm{d}s_1\nonumber\\
	&=\frac{(-\lambda\partial_\lambda)^n}{n!}\int_{0}^{\infty}e^{-\lambda s_1t_2}f_{\alpha_1,\beta_1}(s_1,t_1)\,\mathrm{d}s_1\label{pf1.1}.
\end{align}
On taking the Laplace transform on both sides of (\ref{pf1.1}) with respect to variable $t_1$ and using (\ref{driftlap2}), we get
\begin{align*}
\int_{0}^{\infty}e^{-z_1t_1}	p_{\alpha_1}^{\beta_1}(n,t_1,t_2)\,\mathrm{d}t_1
&=\frac{(-\lambda\partial_\lambda)^n}{n!}z_1^{\beta_1-1}\int_{0}^{\infty}e^{-\lambda s_1t_2}s_1^{\alpha_1-1}E_{\alpha_1,\alpha_1}(-z_1^{\beta_1} s_1^{\alpha_1})\mathrm{d}s_1\\
&=\frac{(-\lambda\partial_\lambda)^nz_1^{\beta_1-1}}{n!(z_1^{\beta_1}+(\lambda t_2)^{\alpha_1})},
\end{align*}
where we have used the following result (see \cite{Kilbas2006}):
\begin{equation*}
	\int_{0}^{\infty}e^{-zt}t^{\beta-1}E_{\alpha,\beta}(ct^\alpha)\,\mathrm{d}t=\frac{z^{\alpha-\beta}}{z^\alpha-c},\ \alpha>0,\ \beta\in\mathbb{R},\ z>|c|^{1/\alpha},\ c\in\mathbb{R}.
\end{equation*}
Its inverse Laplace transform with respect to variable $z_1$ yields the required distribution. 

Now, on applying the Caputo fractional derivative operator $\mathcal{D}^{\beta_1}_{t_1}$ on both sides of (\ref{apf1}) and using the fact that it commutes with integral, we get
\begin{align}
	\mathcal{D}^{\beta_1}_{t_1}p_{\alpha_1}^{\beta_1}(n,t_1,t_2)&=\int_{0}^{\infty}p(n,s_1,t_2)\mathcal{D}^{\beta_1}_{t_1}f_{\alpha_1,\beta_1}(s_1,t_1)\,\mathrm{d}s_1\nonumber\\
	&=-\int_{0}^{\infty}p(n,s_1,t_2)\partial^{\alpha_1}_{s_1}f_{\alpha_1,\beta_1}(s_1,t_1)\,\mathrm{d}s_1,\ \ (\text{using (\ref{driftequ})}),\nonumber\\
	&=\frac{-1}{\Gamma(1-\alpha_1)}\int_{0}^{\infty}p(n,s_1,t_2)\bigg(\frac{\partial}{\partial s_1}\int_{0}^{s_1}(s_1-y)^{-\alpha_1}f_{\alpha_1,\beta_1}(y,t_1)\,\mathrm{d}y\bigg)\,\mathrm{d}s_1\nonumber\\
	&=\frac{-1}{\Gamma(1-\alpha_1)}\int_{0}^{\infty}\frac{\partial}{\partial s_1}\bigg(p(n,s_1,t_2)\int_{0}^{s_1}(s_1-y)^{-\alpha_1}f_{\alpha_1,\beta_1}(y,t_1)\,\mathrm{d}y\bigg)\,\mathrm{d}s_1\nonumber\\
	&\ \ +\frac{1}{\Gamma(1-\alpha_1)}\int_{0}^{\infty}\bigg(\int_{0}^{s_1}(s_1-y)^{-\alpha_1}f_{\alpha_1,\beta_1}(y,t_1)\,\mathrm{d}y\bigg)\frac{\partial}{\partial s_1}p(n,s_1,t_2)\,\mathrm{d}s_1\nonumber\\
	&=\frac{1}{\Gamma(1-\alpha_1)}\int_{0}^{\infty}\bigg(\int_{0}^{s_1}(s_1-y)^{-\alpha_1}f_{\alpha_1,\beta_1}(y,t_1)\,\mathrm{d}y\bigg)\frac{\partial}{\partial s_1}p(n,s_1,t_2)\,\mathrm{d}s_1.\label{tprfeqpf1}
\end{align}
From (\ref{prfdist}), it follows that 
\begin{equation}\label{prfeq}
	\frac{\partial}{\partial s_1}p(n,s_1,t_2)=-\lambda t_2(I-B)p(n,s_1,t_2),\ n\ge0,
\end{equation}
where $p(n,t_1,t_2)=0$ for all $n<0$, and $I$ is the identity operator and $B$ is the backward shift operator such that $Bp(n,t_1,t_2)=p(n-1,t_1,t_2)$.

On substituting (\ref{prfeq}) in (\ref{tprfeqpf1}), we get
\begin{align*}
	\mathcal{D}^{\beta_1}_{t_1}&p_{\alpha_1}^{\beta_1}(n,t_1,t_2)\nonumber\\
	&=\frac{-\lambda t_2}{\Gamma(1-\alpha_1)}\int_{0}^{\infty}\bigg(\int_{0}^{s_1}(s_1-y)^{-\alpha_1}f_{\alpha_1,\beta_1}(y,t_1)\,\mathrm{d}y\bigg)(I-B)p(n,s_1,t_2)\,\mathrm{d}s_1\nonumber\\
	&=\frac{-\lambda t_2}{\Gamma(1-\alpha_1)}\int_{0}^{\infty}(I-B)\bigg(\int_{y}^{\infty}(s_1-y)^{-\alpha_1}p(n,s_1,t_2)\,\mathrm{d}s_1\bigg)f_{\alpha_1,\beta_1}(y,t_1)\,\mathrm{d}y\nonumber\\
	&=\frac{-\lambda t_2}{\Gamma(1-\alpha_1)}\int_{0}^{\infty}(I-B)\bigg(\int_{0}^{\infty}w^{-\alpha_1}p(n,w+y,t_2)\,\mathrm{d}w\bigg)f_{\alpha_1,\beta_1}(y,t_1)\,\mathrm{d}y\nonumber\\
	&=\frac{-\lambda t_2}{\Gamma(1-\alpha_1)}\int_{0}^{\infty}(I-B)\bigg(\int_{0}^{\infty}w^{-\alpha_1}\frac{(\lambda(w+y)t_2)^n}{n!}e^{-\lambda(w+y)t_2}\,\mathrm{d}w\bigg)f_{\alpha_1,\beta_1}(y,t_1)\,\mathrm{d}y\nonumber\\
	&=\frac{-(\lambda t_2)^{n+1}}{\Gamma(1-\alpha_1)}\int_{0}^{\infty}(I-B)\sum_{r=0}^{n}\frac{y^{n-r}}{r!(n-r)!}e^{-\lambda yt_2}\bigg(\int_{0}^{\infty}w^{r-\alpha_1}e^{-\lambda w t_2}\,\mathrm{d}w\bigg)f_{\alpha_1,\beta_1}(y,t_1)\,\mathrm{d}y\nonumber\\
	&=\frac{-(\lambda t_2)^{\alpha_1}}{\Gamma(1-\alpha_1)}\int_{0}^{\infty}(I-B)\sum_{r=0}^{n}\frac{(y\lambda t_2)^{n-r}\Gamma(r-\alpha_1+1)}{r!(n-r)!}e^{-\lambda yt_2}f_{\alpha_1,\beta_1}(y,t_1)\,\mathrm{d}y\nonumber\\
	&=-(\lambda t_2)^{\alpha_1}\int_{0}^{\infty}(I-B)\sum_{r=0}^{n}\frac{\Gamma(r-\alpha_1+1)}{r!\Gamma(1-\alpha_1)}p(n-r,y,t_2)f_{\alpha_1,\beta_1}(y,t_1)\,\mathrm{d}y\nonumber\\
	&=-(\lambda t_2)^{\alpha_1}(I-B)\sum_{r=0}^{n}\frac{\Gamma(r-\alpha_1+1)}{r!\Gamma(1-\alpha_1)}p_{\alpha_1}^{\beta_1}(n-r,t_1,t_2)\nonumber\\
	&=-(\lambda t_2)^{\alpha_1}(I-B)\sum_{r=0}^{\infty}\frac{\Gamma(r-\alpha_1+1)}{r!\Gamma(1-\alpha_1)}B^rp_{\alpha_1}^{\beta_1}(n,t_1,t_2)\\
	&=-(\lambda t_2)^{\alpha_1}(I-B)\sum_{r=0}^{\infty}(-1)^r\frac{(\alpha_1-1)_r}{r!}B^rp_{\alpha_1}^{\beta_1}(n,t_1,t_2),
\end{align*}
where we have used the following identity to get the last step:
\begin{equation*}
	\Gamma(x-r)=(-1)^{r-1}\frac{\Gamma(-x)\Gamma(x+1)}{\Gamma(r-x+1)},\ x\notin\mathbb{Z},\ r\in\mathbb{Z}.
\end{equation*}
Here, $\mathbb{Z}$ denotes the set of integers.
Finally, on using (\ref{grnbino}), the proof completes.
\end{proof}
\begin{remark}
	Alternatively, the system of equations (\ref{fprftc1}) can be obtained by using the governing differential equation of the pgf of $\mathscr{N}_{\alpha_1,\beta_1}^{1,1}(t_1,t_2)$ as follows: 
	\begin{align*}
		G_{\alpha_1}^{\beta_1}(u,t_1,t_2)=\mathbb{E}u^{\mathscr{N}_{\alpha_1,\beta_1}^{1,1}(t_1,t_2)}&=\sum_{n=0}^{\infty}u^n\frac{(-\lambda\partial_\lambda)^n}{n!}E_{\beta_1,1}(-(\lambda t_2)^{\alpha_1}t_1^{\beta_1})\\
		&=E_{\beta_1,1}(-(\lambda t_2)^{\alpha_1}t_1^{\beta_1})\exp(-\lambda u\partial_\lambda)\\
		&=\exp(-\lambda u\partial_\lambda)\mathbb{E}\exp(-\lambda t_2H^{\alpha_1,\beta_1}(t_1))\\
		&=\mathbb{E}\exp(-\lambda t_2(1-u)H^{\alpha_1,\beta_1}(t_1))\\
		&=E_{\beta_1,1}(-\lambda^{\alpha_1}t_2^{\alpha_1}(1-u)^{\alpha_1}t_1^{\beta_1}),\ (t_1,t_2)\in\mathbb{R}^2_+,\ |u|\leq1,
	\end{align*}
	where we have used $e^{c\partial_x}f(x)=f(x+c)$.
	It is the solution of the following fractional differential equation:
	\begin{equation}\label{tprfpgfeq}
		\mathcal{D}_{t_1}^{\beta_1}G_{\alpha_1}^{\beta_1}(u,t_1,t_2)=-\lambda^{\alpha_1}t_2^{\alpha_1}(1-u)^{\alpha_1}G_{\alpha_1}^{\beta_1}(u,t_1,t_2),
	\end{equation}
	with $G_{\alpha_1}^{\beta_1}(u,0,t_2)=1$. On substituting $G_{\alpha_1}^{\beta_1}(u,t_1,t_2)=\sum_{n=0}^{\infty}u^np_{\alpha_1}^{\beta_1}(n,t_1,t_2)$ on both sides of (\ref{tprfpgfeq}), we get
	\begin{align}
		\sum_{n=0}^{\infty}u^n\mathcal{D}_{t_1}^{\beta_1}p_{\alpha_1}^{\beta_1}(n,t_1,t_2)&=-(\lambda t_2)^{\alpha_1}\sum_{l=0}^{\infty}\frac{(\alpha_1)_l}{l!}(-u)^l\sum_{n=0}^{\infty}u^np_{\alpha_1}^{\beta_1}(n,t_1,t_2)\nonumber\\
		&=-(\lambda t_2)^{\alpha_1}\sum_{k=0}^{\infty}u^k\sum_{l=0}^{k}\frac{(\alpha_1)_l}{l!}(-1)^lp_{\alpha_1}^{\beta_1}(k-l,t_1,t_2)\nonumber\\
		&=-(\lambda t_2)^{\alpha_1}\sum_{l=0}^{\infty}\sum_{k=0}^{\infty}u^k\frac{(\alpha_1)_l}{l!}(-1)^lp_{\alpha_1}^{\beta_1}(k-l,t_1,t_2)\nonumber\\
		&=-(\lambda t_2)^{\alpha_1}\sum_{l=0}^{\infty}\sum_{k=0}^{\infty}u^k\frac{(\alpha_1)_l}{l!}(-B)^lp_{\alpha_1}^{\beta_1}(k,t_1,t_2)\nonumber\\
		&=-(\lambda t_2)^{\alpha_1}\sum_{k=0}^{\infty}u^k(I-B)^{\alpha_1}p_{\alpha_1}^{\beta_1}(k,t_1,t_2).\label{coefequs}
	\end{align}
	On comparing the coefficients of $u^n$ on both sides of (\ref{coefequs}), we get the required system of governing differential equations given in (\ref{fprftc1}). 
\end{remark}
\begin{remark}
		For $\alpha_1=\alpha_2=\beta_1=1$, the random field (\ref{tprf}) reduces to a time-changed Poisson random field  (see \cite{Kataria2024}, Remark 4.4). Also, the distribution (\ref{dist1f}) reduces to
		\begin{equation*}
			p_{1,1}^{1,\beta_2}(n,t_1,t_2)=\frac{(-\lambda\partial_\lambda)^n}{n!}E_{\beta_2,1}(-\lambda t_1t_2^{\beta_2}),\ n\ge0,
		\end{equation*}
		which solves the following system of fractional differential equations:
		\begin{equation*}
			\mathcal{D}^{\beta_1}_{t_1}p_{1,1}^{1,\beta_2}(n,t_1,t_2)=-\lambda t_2(I-B)p_{1,1}^{1,\beta_2}(n,t_1,t_2),\ n\ge0,\ t_2\ge0.
		\end{equation*} 
Further, for $\alpha_2=\beta_1=\beta_2=1$, the time-changed random field defined in (\ref{tprf}) reduces to $\mathscr{N}_{\alpha_1,1}^{1,1}(t_1,t_2)\coloneqq\mathscr{N}(S_{\alpha_1}(t_1),t_2)$, $(t_1,t_2)\in\mathbb{R}^2_+$. Its one dimensional distribution is given by
\begin{equation*}
	p_{\alpha_1,1}^{1,1}(n,t_1,t_2)=\frac{(-\lambda\partial_\lambda)^n}{n!}e^{-(\lambda t_2)^{\alpha_1}t_1},\ n\ge0,
\end{equation*}
which solves 
\begin{equation*}
	\partial_{t_1}p_{\alpha_1,1}^{1,1}(n,t_1,t_2)=-\lambda^{\alpha_1}t_2^{\alpha_1}(I-B)^{\alpha_1}p_{\alpha_1}(n,t_1,t_2),\ n\ge0,\ t_2\ge0.
\end{equation*}
\end{remark}
\begin{remark}
	For $\alpha_1=\beta_1=1$, the distribution of (\ref{tprf}) is 
	\begin{equation*}
		p^{\alpha_2,\beta_2}(n,t_1,t_2)=\mathrm{Pr}\{N_{1,1}^{\alpha_2,\beta_2}(t_1,t_2)=n\}=\frac{(-\lambda\partial_\lambda)^n}{n!}E_{\beta_2,1}(-(\lambda t_1)^{\alpha_2}t_1^{\beta_2}),\ n\ge0,\ (t_1,t_2)\in\mathbb{R}^2_+,
	\end{equation*}
	which solves 
	\begin{equation*}
		\mathcal{D}_{t_2}^{\beta_2}p^{\alpha_2,\beta_2}(n,t_1,t_2)=-\lambda^{\alpha_2}t_1^{\alpha_2}(1-B)^{\alpha_2}p^{\alpha_2,\beta_2}(n,t_1,t_2),\ n\ge0,\ t_1\ge0,
	\end{equation*}
	with $p^{\alpha_2,\beta_2}(0,t_1,0)=1$. Its proof follows similar lines to that of Theorem \ref{thmtc1}. Hence, it is omitted.
\end{remark}

\section{Poisson random fields with drifts}\label{sec4}
In this section, we study PRFs with deterministic and random drifts. First, we consider the PRF on $\mathbb{R}^d_+$, $d\ge1$ with a constant drift. Let $\{\mathscr{N}(B),\ B\in\mathcal{B}\}$ be the PRF as defined in Section \ref{sec1.1}. Then, the PRF with constant drift is defined as follows:
\begin{equation}\label{rfd}
	\mathscr{N}(B)+a|B|,\ a>0,\ B\in\mathcal{B}_{d}.
\end{equation}
Its Laplace transform is given by
\begin{equation}\label{lapd=2}
	\mathbb{E}e^{-\eta\mathscr{N}(B)-a\eta|B|}=e^{-\lambda|B|}\sum_{k=0}^{\infty}\frac{(\lambda|B|)^k}{k!}e^{-\eta (a|B|+k)},\ \eta>0,
\end{equation}
whose inversion yields 
\begin{equation*}
	\mathrm{Pr}\{\mathscr{N}(B)+a|B|\in\mathrm{d}x\}/\mathrm{d}x=e^{-\lambda|B|}\sum_{k=0}^{\infty}\frac{(\lambda|B|)^k}{k!}\delta(x-a|B|-k),\ x\ge a|B|,
\end{equation*}
where $\delta(\cdot)$ is the Dirac delta function. Here, we have used the Laplace transform $\int_{0}^{\infty}e^{-\eta x}\delta(x-c)\mathrm{d}x=e^{-\eta c}$, $\eta >0$, $c>0$ of the Dirac delta function to obtain the last step (see \cite{Gradshteyn2007}).
\begin{remark}
	Let $\{\mathscr{N}(t),\ t\ge0\}$ be the homogeneous Poisson process with parameter $\lambda>0$. Then,  for $d=1$, the random field defined by (\ref{rfd}) reduces to the Poisson process with constant drift $\mathscr{N}(t)+at$, $t\ge0$, introduced and studied in \cite{Beghin2014}. Its distribution is given by
	\begin{equation*}
		p(x,t)=e^{-\lambda t}\sum_{k=0}^{\infty}\frac{(\lambda t)^k}{k!}\delta(x-at-k),\ x\ge at,\ t\ge0.
	\end{equation*}
\end{remark}
\subsection{PRF on $\mathbb{R}^2_+$ with random drift} Here, we study the Poisson random field on $\mathbb{R}^2_+$ with random drift. First, we consider the PRF $\{\mathscr{N}(t_1,t_2),\ (t_1,t_2)\in\mathbb{R}^2_+\}$ as defined in Section \ref{sec1.1}. For $d=2$, the random field defined in (\ref{rfd}) reduces to a PRF with constant drift $t_1t_2$ 
\begin{equation}\label{dprfd=2}
	\mathscr{N}(t_1,t_2)+at_1t_2,\ a>0,\ (t_1,t_2)\in\mathbb{R}^2_+.
\end{equation}
Its density function is given by
\begin{equation}\label{d=2}
	\mathrm{Pr}\{\mathscr{N}(t_1,t_2)+at_1t_2\in\mathrm{d}x\}/\mathrm{d}x=e^{-\lambda t_1t_2}\sum_{k=0}^{\infty}\frac{(\lambda t_1t_2)^k}{k!}\delta(x-a t_1t_2-k),\ x\ge at_1t_2,\ (t_1,t_2)\in\mathbb{R}^2_+.
\end{equation} 

Next result shows that the PRF with constant drift has stationary and independent increments as in Definition \ref{instain}.
\begin{proposition}
	The random field defined in (\ref{dprfd=2}) has stationary and independent rectangular increments.
\end{proposition}
\begin{proof}
	For $(h,k)$ and $(s,t)$ in $\mathbb{R}^2_+$, the increment of (\ref{dprfd=2}) over rectangle $R=(s,s+h]\times(t,t+k]$ is given by
	\begin{align*}
		\mathscr{N}(s+h,t+k)&-\mathscr{N}(s,t+k)-\mathscr{N}(s+h,t)-\mathscr{N}(s,t)+a((s+h)(t+k)-s(t+k)-(s+h)t+st)\\
		&=\mathscr{N}(s+h,t+k)-\mathscr{N}(s,t+k)-\mathscr{N}(s+h,t)-\mathscr{N}(s,t)+ahk\\
		&\overset{d}{=}\mathscr{N}(h,k)+ahk,
	\end{align*}
	where we have used the stationary increments property of the PRF. This prove that (\ref{dprfd=2}) has stationary increments.
	
	Let $R_1,R_2,\dots,R_m$ be disjoint rectangles in $\mathbb{R}^2_+$ of type $R$. Then, the increments of (\ref{dprfd=2}) over these rectangles are $\mathscr{N}(R_1)+|R_1|, \mathscr{N}(R_2)+|R_2|,\dots,\mathscr{N}(R_m)+|R_m|$, respectively, where $|R_i|$ denotes the area of rectangle $R_i$ for each $i=1,2,\dots,m$. Thus, the independent increments property of (\ref{dprfd=2}) follows from the independent increments of the PRF $\{\mathscr{N}(t_1,t_2),\ (t_1,t_2)\in\mathbb{R}^2_+\}$.
\end{proof}

\begin{proposition}
	For $i=1,2$, the density \eqref{d=2} solves the following differential equations:
	\begin{equation}\label{deg1}
		(\partial_{t_1}+at_2\partial_x+\lambda t_2(I-e^{-\partial_x}))w(x,t_1,t_2)=0,
	\end{equation}
	and 
	\begin{equation}\label{deg2}
		(\partial_{t_2}+at_1\partial_x+\lambda t_1(I-e^{-\partial_x}))w(x,t_1,t_2)=0,
	\end{equation}
	with initial conditions $w(x,0,t_2)=w(x,t_1,0)=\delta(x)$ and boundary condition $w(0,t_1,t_2)=0$. Here, $e^{\partial_x}$ is a shift operator such that $e^{c\partial_x}f(x)=f(x+c)$, $c\in\mathbb{R}$.
\end{proposition}
\begin{proof}
	Let $\tilde{w}(\eta, t_1,t_2)=\int_{0}^{\infty}e^{-\eta x}w(x,t_1,t_2)\, \mathrm{d}x$, $\eta>0$. Then, on taking the Laplace transform with respect to $x$ on both sides of \eqref{deg1}, we get
	\begin{equation}\label{cpf1}
		\partial_{t_1}\tilde{w}(\eta, t_1,t_2)+at_2\eta \tilde{w}(\eta, t_1,t_2)+\lambda t_2(1-e^{-\eta})\tilde{w}(\eta, t_1,t_2)=0,
	\end{equation}
	where we have used 
	\begin{align*}
		\int_{0}^{\infty}e^{-\eta x}(I-e^{-\partial_x})w(x,t_1,t_2)\,\mathrm{d}x&=\int_{0}^{\infty}e^{-\eta x}w(x,t_1,t_2)\,\mathrm{d}x-\int_{0}^{\infty}e^{-\eta x}w(x-1,t_1,t_2)\,\mathrm{d}x\\
		&=\tilde{w}(\eta, t_1,t_2)(1-e^{-\eta}).
	\end{align*}
	On solving (\ref{cpf1}) with $\tilde{w}(\eta, 0,t_2)=1$, we get
	\begin{equation}\label{laptil}
		\tilde{w}(\eta, t_1,t_2)=\exp(-at_1t_2-\lambda t_1t_2(1-e^{-\eta})),\ \eta>0.
	\end{equation}
	
	The Laplace transform  \eqref{d=2} with respect to $x$ can be obtained by taking $d=2$ in \eqref{lapd=2} and it is given by 
	\begin{equation*}
		\int_{0}^{\infty}e^{-\eta x}\mathrm{Pr}\{\mathscr{N}(t_1,t_2)+at_1t_2\in\mathrm{d}x\}=\exp(-\eta a t_1t_2-\lambda t_1 t_2(1-e^{-\eta})),
	\end{equation*}
	which coincides with (\ref{laptil}).
	Thus, on using the uniqueness of Laplace transform, the density function \eqref{d=2} solves \eqref{deg1}. Similarly, it can be established that it is a solution of \eqref{deg2}.  This completes the proof.
\end{proof}

\begin{remark}
	Let $f(\cdot)$ be an absolutely integrable function with compact support on $(0,\infty)$. Consider a two parameter operator
	\begin{equation*}
		P_{t_1,t_2}f(x)=\mathbb{E}f(x-\mathscr{N}(t_1,t_2)-at_1t_2)=e^{-\lambda t_1t_2}\sum_{k=0}^{\infty}\frac{(\lambda t_1t_2)^k}{k!}f(x-at_1t_2-k),\ x>0,
	\end{equation*}
	whose Laplace transform is 
	$
	\tilde{P}_{t_1,t_2}f(\eta)=\int_{0}^{\infty}e^{-\eta x}P_{t_1,t_2}f(x)\,\mathrm{d}x=\tilde{f}(\eta)\tilde{p}(\eta,t_1,t_2).
	$
	In particular, if $f(x)=\delta(x)$ then $p(x,t_1,t_2)=P_{t_1,t_2}\delta(x)$. Also, for $s_1\ge0$, we have
	\begin{align*}
		P_{s_1+t_1,t_2}f(x)&=e^{-\lambda (s_1+t_1)t_2}\sum_{k=0}^{\infty}\frac{(\lambda (s_1+t_1)t_2)^k}{k!}f(x-as_1t_2-at_1t_2-k)\\
		&=e^{-\lambda (s_1+t_1)t_2}\sum_{k=0}^{\infty}(\lambda t_2)^k\sum_{r=0}^{k}\frac{s_1^rt_1^{k-r}}{r!(k-r)!}f(x-as_1t_2-at_1t_2-k)\\
		&=e^{-\lambda (s_1+t_1)t_2}\sum_{r=0}^{\infty}\sum_{k=r}^{\infty}\frac{(\lambda s_1t_2)^r(\lambda t_1t_2)^{k-r}}{r!(k-r)!}f(x-as_1t_2-at_1t_2-k)\\
		&=e^{-\lambda (s_1+t_1)t_2}\sum_{r=0}^{\infty}\sum_{k=0}^{\infty}\frac{(\lambda s_1t_2)^r(\lambda t_1t_2)^k}{r!k!}f(x-as_1t_2-at_1t_2-k-r)\\
		&=e^{-\lambda s_1t_2}\sum_{r=0}^{\infty}\frac{(\lambda s_1t_2)^r}{r!}P_{t_1,t_2}f(x-as_1t_2-r)=P_{s_1,t_2}P_{t_1,t_2}f(x).
	\end{align*}
	Similarly, we can show that $P_{t_1,s_2+t_2}=P_{t_1,s_2}P_{t_1,t_2}$ for all $s_2\ge0$. Therefore, $P_{t_1,t_2}$ is a coordinatewise semigroup operator.
\end{remark}

\subsection{Time-changed PRF with random drift} Here, we study different types of time-changed PRF with random drifts. First, we consider the PRF time-changed with independent inverse stable subordinators. 

\paragraph{\textbf{Type I}}Let $\{L_{\beta_i}(t),\ t\ge0\}$, $\beta_i\in(0,1]$, $i=1,2$ be two independent inverse stable subordinators. Let $\{\mathscr{N}(t_1,t_2),\ (t_1,t_2)\in\mathbb{R}^2_+\}$ be a PRF with parameter $\lambda>0$ that is independent of $\{(L_{\beta_1}(t_1),L_{\beta_2}(t_2)),\ (t_1,t_2)\in\mathbb{R}^2_+\}$. Then, we consider the following time-changed PRF with random drift:
\begin{equation}\label{type1.1}
	\mathscr{N}(L_{\beta_1}(t_1),L_{\beta_2}(t_2))+aL_{\beta_1}(t_1)L_{\beta_2}(t_2),\ (t_1,t_2)\in\mathbb{R}^2_+,\ a>0.
\end{equation}
\begin{proposition}
	The Laplace transform of random field (\ref{type1.1}) is given by
	\begin{multline}\label{lapd1}
		\mathbb{E}\exp\big(-\eta\big(	\mathscr{N}(L_{\beta_1}(t_1),L_{\beta_2}(t_2))+aL_{\beta_1}(t_1)L_{\beta_2}(t_2)\big)\big)\\={}_2\Psi_2\left[\begin{matrix}
			(1,1)&(1,1)\\\\
			(1,\beta_1)&(1,\beta_2)
		\end{matrix}\Bigg|-(\lambda(1-e^{-\eta})+a\eta)t_1^{\beta_1}t_2^{\beta_2}\right],\ \eta>0,
	\end{multline}
	where ${}_2\Psi_2$ is the generalized Wright function defined in (\ref{genwrit}).
\end{proposition}
\begin{proof}
	For $(t_1,t_2)\in\mathbb{R}^2_+$, we have
	\begin{align*}
		\mathbb{E}\exp\big(-\eta\big(	\mathscr{N}(L_{\beta_1}&(t_1),L_{\beta_2}(t_2))+aL_{\beta_1}(t_1)L_{\beta_2}(t_2)\big)\big)\\
		&=\mathbb{E}\big(\mathbb{E}\exp\big(-\eta\big(	\mathscr{N}(L_{\beta_1}(t_1),L_{\beta_2}(t_2))+aL_{\beta_1}(t_1)L_{\beta_2}(t_2)\big)\big)|(L_{\beta_1}(t_1),L_{\beta_2}(t_2))\big)\\
		&=\mathbb{E}\exp(\lambda L_{\beta_1}(t_1)L_{\beta_2}(t_2)(e^{-\eta}-1)-a\eta L_{\beta_1}(t_1)L_{\beta_2}(t_2))\\
		&=\mathbb{E}(\mathbb{E}\exp(L_{\beta_1}(t_1)L_{\beta_2}(t_2)(\lambda (e^{-\eta}-1)-a\eta))|L_{\beta_2}(t_2))\\
		&=\mathbb{E}E_{\beta_1,1}(-(\lambda(1-e^{-\eta})+a\eta)L_{\beta_2}(t_2)t_1^{\beta_1})\\
		&=\int_{0}^{\infty}E_{\beta_1,1}(-(\lambda(1-e^{-\eta})+a\eta)x_2t_1^{\beta_1})\mathrm{Pr}\{L_{\beta_2}(t_2)\in\mathrm{d}x_2\}.
	\end{align*}
	Its Laplace transform with respect to $t_2$ is given by
	\begin{align*}
		\int_{0}^{\infty}e^{-z_2t_2}\mathbb{E}\exp\big(-\eta\big(	\mathscr{N}(L_{\beta_1}(t_1)&,L_{\beta_2}(t_2))+aL_{\beta_1}(t_1)L_{\beta_2}(t_2)\big)\big)\,\mathrm{d}t_2\\
		&=z_2^{\beta_2-1}\int_{0}^{\infty}E_{\beta_1,1}(-(\lambda(1-e^{-\eta})+a\eta)x_2t_1^{\beta_1})e^{-z_2^{\beta_2}x_2}\,\mathrm{d}x_2\\
		&=z_2^{\beta_2-1}\sum_{k=0}^{\infty}\frac{(-(\lambda(1-e^{-\eta})+a\eta)t_1^{\beta_1})^k}{\Gamma(k\beta_1+1)}\int_{0}^{\infty}x_2^{k}e^{-z_2^{\beta_2}x_2}\,\mathrm{d}x_2\\
		&=\sum_{k=0}^{\infty}\frac{(-(\lambda(1-e^{-\eta})+a\eta)t_1^{\beta_1})^k}{\Gamma(k\beta_1+1)}\frac{\Gamma(k+1)}{z_2^{k\beta_2+1}},
	\end{align*}
	whose inversion with respect to $z_2$ yields
	\begin{equation*}
		\mathbb{E}\exp\big(-\eta\big(	\mathscr{N}(L_{\beta_1}(t_1),L_{\beta_2}(t_2))+aL_{\beta_1}(t_1)L_{\beta_2}(t_2)\big)\big)=\sum_{k=0}^{\infty}\frac{\Gamma(k+1)(-(\lambda(1-e^{-\eta})+a\eta)t_1^{\beta_1}t_2^{\beta_2})^k}{\Gamma(k\beta_1+1)\Gamma(k\beta_2+1)}.
	\end{equation*}
	This completes the proof.
\end{proof}
\begin{remark}
	For $\beta_1=\beta_2=1$, the Laplace transform (\ref{lapd1}) reduces to 
	\begin{align*}
		\mathbb{E}\exp\big(-\eta\mathscr{N}(t_1,t_2)-\eta at_1t_2\big)&=\sum_{k=0}^{\infty}\frac{(-(\lambda(1-e^{-\eta})+a\eta)t_1t_2)^k}{\Gamma(k+1)}\\
		&=\exp\big(-(\lambda(1-e^{-\eta})+a\eta)t_1t_2\big)\\
		&=e^{-a\eta t_1t_2}\sum_{k=0}^{\infty}\frac{(-\lambda t_1t_2)^k}{k!}(1-e^{-\eta})^k\\
		&=e^{-a\eta t_1t_2}\sum_{k=0}^{\infty}\frac{(-\lambda t_1t_2)^k}{k!}\sum_{r=0}^{k}\binom{k}{r}(-1)^re^{-\eta r}\\
		&=e^{-a\eta t_1t_2}\sum_{r=0}^{\infty}\frac{(\lambda t_1t_2)^re^{-\eta r}}{r!}\sum_{k=r}^{\infty}\frac{(-\lambda t_1t_2)^{k-r}}{(k-r)!}\\
		&=e^{-\lambda t_1t_2}\sum_{r=0}^{\infty}\frac{(\lambda t_1t_2)^r}{r!}e^{-\eta(at_1t_2+r)}.
	\end{align*}
	Its inverse Laplace transform coincides with (\ref{d=2}).
\end{remark}

\paragraph{\textbf{Type II}}Let $\{S_{\alpha}(t_1),\ t_1\ge0\}$, $0<\alpha\leq1$ and $\{L_{\beta}(t_2),\ t_2\ge0\}$, $0<\beta\leq1$ be independent stable and inverse stable subordinators, respectively. Let us consider the following time-changed PRF with random drift:
$\mathscr{N}(S_{\alpha}(t_1),L_{\beta}(t_2))+aS_{\alpha}(t_1)L_{\beta}(t_2)$, $(t_1,t_2)\in\mathbb{R}^2_+$,
where $\mathscr{N}(t_1,t_2)$ and $(S_{\alpha}(t_1),L_{\beta}(t_2))$ are independent of each other. Its Laplace transform is given by
\begin{align*}
	\mathbb{E}\exp\big(-\eta\big(	\mathscr{N}(S_{\alpha}(t_1),&L_{\beta}(t_2))+aS_{\alpha}(t_1)L_{\beta}(t_2)\big)\big)\\
	&=\mathbb{E}\big(\mathbb{E}\exp\big(-\eta\big(	\mathscr{N}(S_{\alpha}(t_1),L_{\beta}(t_2))+aS_{\alpha}(t_1)L_{\beta}(t_2)\big)\big)|(S_{\alpha}(t_1),L_{\beta}(t_2))\big)\\
	&=\mathbb{E}\exp(\lambda S_{\alpha}(t_1)L_{\beta}(t_2)(e^{-\eta}-1)-a\eta S_{\alpha}(t_1)L_{\beta}(t_2))\\
	&=\mathbb{E}(\mathbb{E}\exp(S_{\alpha}(t_1)L_{\beta}(t_2)(\lambda (e^{-\eta}-1)-a\eta))|L_{\beta}(t_2))\\
	&=\int_{0}^{\infty}\exp(-(x(\lambda (1-e^{-\eta})+a\eta))^\alpha t_1)\mathrm{Pr}\{L_\beta(t_2)\in\mathrm{d}x\},
\end{align*}
which on taking Laplace transform with respect to $t_2$ yields
\begin{align*}
	\int_{0}^{\infty}e^{-z_2t_2}\mathbb{E}\exp\big(-\eta\big(	\mathscr{N}(S_{\alpha}(t_1),&L_{\beta}(t_2))+aS_{\alpha}(t_1)L_{\beta}(t_2)\big)\big)\,\mathrm{d}t_2\\
	&=z_2^{\beta-1}\int_{0}^{\infty}\exp(-(x(\lambda (1-e^{-\eta})+a\eta))^\alpha t_1)e^{-z_2^\beta x}\,\mathrm{d}x\\
	&=z_2^{\beta-1}\sum_{k=0}^{\infty}\frac{(-t_1)^k(\lambda (1-e^{-\eta})+a\eta)^{\alpha k}}{k!}\int_{0}^{\infty}x^{\alpha k}e^{-z_2^\beta x}\,\mathrm{d}x\\
	&=\sum_{k=0}^{\infty}\frac{(-t_1)^k(\lambda (1-e^{-\eta})+a\eta)^{\alpha k}\Gamma(\alpha k+1)}{k!z_2^{\alpha\beta k+1}},\ \eta>0,\ z_2>0.
\end{align*}
Its inverse Laplace transform with respect to $z_2$ is given by
\begin{align*}
	\mathbb{E}\exp\big(-\eta\big(	\mathscr{N}(S_{\alpha}(t_1),L_{\beta}(t_2))+aS_{\alpha}(t_1)L_{\beta}(t_2)\big)\big)&=\sum_{k=0}^{\infty}\frac{(-t_1t_2^{\alpha\beta})^k(\lambda (1-e^{-\eta})+a\eta)^{\alpha k}\Gamma(\alpha k+1)}{k!\Gamma(\alpha\beta k+1)}\\
	&={}_1\Psi_1\left[\begin{matrix}
		(1,\alpha)\\\\
		(1,\alpha\beta)
	\end{matrix}\Bigg|-t_1t_2^{\alpha\beta}(\lambda (1-e^{-\eta})+a\eta)^{\alpha}\right].
\end{align*}

\paragraph{\textbf{Type III}} For $0<\alpha\leq1$ and $0<\beta\leq1$, let us consider the time-changed PRF with random drift given as follows:
\begin{equation}\label{prfcompdrift}
	\mathscr{N}(H^{\alpha,\beta}(t_1),t_2)+aH^{\alpha,\beta}(t_1)t_2,\ a\ge0,\  (t_1,t_2)\in\mathbb{R}^2_+,
\end{equation}
where $H^{\alpha,\beta}(t)$ is as defined in (\ref{driftdef}) and it is independent of the PRF $\{\mathscr{N}(t_1,t_2),\ (t_1,t_2)\in\mathbb{R}^2_+\}$.

The Laplace transform of (\ref{prfcompdrift}) is given by
\begin{equation*}
	\mathbb{E}\exp(-\eta\mathscr{N}(H^{\gamma,\beta}(t_1),t_2)-\eta aH^{\alpha,\beta}(t_1)t_2)=\mathbb{E}(\mathbb{E}\exp(-\eta\mathscr{N}(H^{\gamma,\beta}(t_1),t_2)-\eta aH^{\alpha,\beta}(t_1)t_2)|L_\beta(t_1)),
\end{equation*}
where
\begin{align*}
	\mathbb{E}\exp(-\eta\mathscr{N}(H^{\gamma,\beta}(t_1),t_2)|L_\beta(t_1))&=\mathbb{E}(\mathbb{E}(\exp(-\eta\mathscr{N}(H^{\gamma,\beta}(t_1),t_2))|H^{\gamma,\beta}(t_1))|L_\beta(t_1))\\
	&=\mathbb{E}(\exp(\lambda H^{\gamma,\beta}(t_1)t_2(e^{-\eta}-1))|L_\beta(t_1))\\
	&=\exp(-(\lambda t_2)^\gamma(1-e^{-\eta})^\gamma L_\beta(t_1))
\end{align*}
and
\begin{equation*}
	\mathbb{E}(\exp(-\eta aH^{\alpha,\beta}(t_1)t_2)|L_\beta(t_1))=\exp(-(\eta at_2)^\alpha L_\beta(t_1)),\ \eta>0.
\end{equation*}
Thus, on using the Laplace transform of inverse stable subordinator, we have
\begin{align}
	\mathbb{E}\exp(-\eta\mathscr{N}(H^{\gamma,\beta}(t_1),t_2)-\eta aH^{\alpha,\beta}(t_1)t_2)\nonumber
	&=\mathbb{E}\exp(-((\lambda t_2)^\gamma(1-e^{-\eta})^\gamma+(\eta at_2)^\alpha)L_\beta(t_1))\nonumber\\
	&=E_{\beta,1}(-((\lambda t_2)^\gamma(1-e^{-\eta})^\gamma+(\eta at_2)^\alpha)t_1^\beta).\label{lapdrf}
\end{align}
\begin{theorem}
	The probability law $r(x,t_1,t_2)=\mathrm{Pr}\{\mathscr{N}(H^{\gamma,\beta}(t_1),t_2)+aH^{\alpha,\beta}(t_1)t_2\in\mathrm{d}x\}/\mathrm{d}x$ of the process (\ref{prfcompdrift}) is given by
	\begin{equation}\label{pmfcompdrft}
		r(x,t_1,t_2)=\sum_{n=0}^{\infty}\frac{(-\lambda\partial_\lambda)^n}{n!}\int_{0}^{\infty}e^{-s(\lambda t_2)^{\gamma}}g_\alpha(x-n,(at_2)^\alpha s)f_\beta(s,t_1)\,\mathrm{d}s,
	\end{equation}
	where $g_\alpha(\cdot,t)$ and $f_\beta(\cdot,t)$ are the densities of $\alpha$-stable and inverse $\beta$-stable subordinators, respectively. It solves the following differential equation:
	\begin{equation*}
		(\mathcal{D}_t^\beta+(at_2)^\alpha\partial_x^\alpha +(\lambda t_2)^\gamma(I-e^{-\partial_x})^\gamma) r(x,t_1,t_2)=0,
	\end{equation*}
	with initial condition $r(x,0,t_2)=r(x,t_1,0)=\delta(s)$ and boundary condition $r(0,t_1,t_2)=0$. Here, $\mathcal{D}_t^\beta$ and $\partial_x^\alpha$ are the Caputo and Riemann-Liouville fractional derivative operators, respectively, and
	\begin{equation*}
		(I-e^{-\partial_x})^\gamma=\sum_{k=0}^{\infty}(-1)^\gamma\binom{\gamma}{k}e^{-k\partial_x},
	\end{equation*}
	where $e^{-k\partial_x}f(x)=f(x-k)$.
\end{theorem}
\begin{proof}
	For any Borel set $C$ of non-negative real number, we have
	\begin{align*}
		\mathrm{Pr}\{\mathscr{N}(H^{\gamma,\beta}(t_1),t_2)&+aH^{\alpha,\beta}(t_1)t_2\in C\}\\
		&=\sum_{n=0}^{\infty}\mathrm{Pr}\{\mathscr{N}(H^{\gamma,\beta}(t_1),t_2)=n\}\mathrm{Pr}\Big\{H^{\alpha,\beta}(t_1)\in \frac{C-n}{at_2}\Big\}\\
		&=\sum_{n=0}^{\infty}\int_{0}^{\infty}\mathrm{Pr}\{\mathscr{N}(S_{\gamma }(s),t_2)=n\}\mathrm{Pr}\Big\{S_{\alpha}(s)\in \frac{C-n}{at_2}\Big\}f_\beta(s,t_1)\,\mathrm{d}s,
	\end{align*}
	where $f_\beta(\cdot,t_1)$ is the density of $L_\beta(t_1)$.
	On substituting $\alpha_1=\gamma$ and $\beta_1=1$ in \eqref{dist1f}, and using the scaling property of stable subordinator, we get
	\begin{align*}
		\mathrm{Pr}\{\mathscr{N}(H^{\gamma,\beta}(t_1),t_2)&+aH^{\alpha,\beta}(t_1)t_2\in C\}\\
		&=\sum_{n=0}^{\infty}\int_{0}^{\infty}\frac{(-\lambda\partial_\lambda)^n}{n!}e^{-s(\lambda t_2)^{\gamma}}\mathrm{Pr}\Big\{S_{\alpha}({(at_2)}^\alpha s)\in C-n\Big\}f_\beta(s,t_1)\,\mathrm{d}s\\
		&=\sum_{n=0}^{\infty}\int_{0}^{\infty}\int_{C}\frac{(-\lambda\partial_\lambda)^n}{n!}e^{-s(\lambda t_2)^{\gamma}}g_\alpha(x-n,(at_2)^\alpha s)f_\beta(s,t_1)\,\mathrm{d}x\,\mathrm{d}s\\
		&=\sum_{n=0}^{\infty}\frac{(-\lambda\partial_\lambda)^n}{n!}\int_{C}\int_{0}^{\infty}e^{-s(\lambda t_2)^{\gamma}}g_\alpha(x-n,(at_2)^\alpha s)f_\beta(s,t_1)\,\mathrm{d}s\,\mathrm{d}x.
	\end{align*}
	So, 
	\begin{equation*}
		\mathrm{Pr}\{\mathscr{N}(H^{\gamma,\beta}(t_1),t_2)+aH^{\alpha,\beta}(t_1)t_2\in \mathrm{d}x\}=\sum_{n=0}^{\infty}\frac{(-\lambda\partial_\lambda)^n}{n!}\int_{0}^{\infty}e^{-s(\lambda t_2)^{\gamma}}g_\alpha(x-n,(at_2)^\alpha s)f_\beta(s,t_1)\,\mathrm{d}s\mathrm{d}x.
	\end{equation*}
	
	To obtain the governing differential equation, let us consider the following double Laplace transform:
	\begin{align*}
		\mathcal{L}(\eta,z_1,t_2)&=\int_{0}^{\infty}\int_{0}^{\infty}e^{-z_1t_1-\eta x}r(x,t_1,t_2)\,\mathrm{d}x\,\mathrm{d}t_1\\
		&=\int_{0}^{\infty}e^{-z_1t_1}E_{\beta,1}(-((\lambda t_2)^\gamma(1-e^{-\eta})^\gamma+(\eta at_2)^\alpha)t_1^\beta)\,\mathrm{d}t_1\\
		&=\frac{z_1^{\beta-1}}{z_1^\beta+(\lambda t_2)^\gamma(1-e^{-\eta})^\gamma+(\eta at_2)^\alpha},
	\end{align*}
	where we have used (\ref{lapdrf}). So, by using $\int_{0}^{\infty}e^{-x\eta}\delta(x)\,\mathrm{d}x=1$, we have
	\begin{equation*}
		z_1^\beta \mathcal{L}(\eta,z_1,t_2)-z_1^{\beta-1}\int_{0}^{\infty}e^{-x\eta}\delta(x)\,\mathrm{d}x+(at_2)^\alpha\eta^\alpha\mathcal{L}(\eta,z_1,t_2)+(\lambda t_2)^\gamma(1-e^{-\eta})^\gamma \mathcal{L}(\eta,z_1,t_2)=0.
	\end{equation*}
Equivalently,
	\begin{align}
		z_1^\beta \mathcal{L}(\eta,z_1,t_2)-z_1^{\beta-1}\int_{0}^{\infty}&e^{-x\eta}\delta(x)\,\mathrm{d}x+(at_2)^\alpha\eta^\alpha\mathcal{L}(\eta,z_1,t_2)\nonumber\\
		&=-(\lambda t_2)^\gamma\sum_{k=0}^{\infty}(-1)^k\binom{\gamma}{k}e^{-k\eta} \mathcal{L}(\eta,z_1,t_2)\nonumber\\
		&=-(\lambda t_2)^\gamma\sum_{k=0}^{\infty}(-1)^k\binom{\gamma}{k}e^{-k\eta}\int_{k}^{\infty}e^{-\eta(x-k)}\tilde{\mathcal{L}}(x-k,z_1,t_2)\,\mathrm{d}x\nonumber\\
		&=-(\lambda t_2)^\gamma\sum_{k=0}^{\infty}(-1)^k\binom{\gamma}{k}\int_{0}^{\infty}e^{-\eta x}e^{-k\partial_x}\tilde{\mathcal{L}}(x,z_1,t_2)\,\mathrm{d}x\nonumber\\
		&=-(\lambda t_2)^\gamma\int_{0}^{\infty}e^{-\eta x}(I-e^{-\partial_x})^\gamma \tilde{\mathcal{L}}(x,z_1,t_2)\,\mathrm{d}x,\label{cpf2}
	\end{align}
	where $\tilde{\mathcal{L}}(x,z_1,t_2)=\int_{0}^{\infty}e^{-z_1t_1}r(x,t_1,t_2)\,\mathrm{d}t_1$. On taking the inverse Laplace transform with respect to $\eta$ on both sides of (\ref{cpf2}) and using $\int_{0}^{\infty}e^{-\eta x}\partial_x^\alpha \tilde{\mathcal{L}}(x,z_1,t_2)\,\mathrm{d}x$ $=\eta^\alpha\mathcal{L}(\eta,z_1,t_2)$, we have
	\begin{equation}\label{cpf3}
		z_1^\beta \tilde{\mathcal{L}}(x,z_1,t_2)-z_1^{\beta-1}r(x,0,t_2)+(at_2)^\alpha\partial_x^\alpha\tilde{\mathcal{L}}(x,z_1,t_2)=-(\lambda t_2)^\gamma(I-B)^\gamma \tilde{\mathcal{L}}(x,z_1,t_2).
	\end{equation}
	Now, on taking the inverse Laplace transform with respect to $z_1$ on both sides of (\ref{cpf3}) and using the following result (see \cite{Kilbas2006}):
	\begin{equation*}
		\int_{0}^{\infty}e^{-zt}\mathcal{D}_t^\beta f(t)\,\mathrm{d}t=z^\beta\int_{0}^{\infty}e^{-zt} f(t)\,\mathrm{d}t-z^{\beta-1}f(0^+),\ z>0,
	\end{equation*} 
	we get 
	\begin{equation*}
		\mathcal{D}_{t_1}^\beta r(x,t_1,t_2)+(at_2)^\alpha\partial_x^\alpha r(x,t_1,t_2)=-(\lambda t_2)^\gamma(I-e^{-\partial_x})^\gamma r(x,t_1,t_2).
	\end{equation*}
	This completes the proof.
\end{proof}
\begin{remark}
	The Laplace transform of the density given in \eqref{pmfcompdrft} is
	\begin{align*}
		\int_{0}^{\infty}e^{-\eta x}\mathrm{Pr}\{\mathscr{N}(H^{\gamma,\beta}(t_1),t_2)&+aH^{\alpha,\beta}(t_1)t_2\in \mathrm{d}x\}\\
		&=\sum_{n=0}^{\infty}\frac{(-\lambda\partial_\lambda)^n}{n!}\int_{0}^{\infty}\int_{0}^{\infty}e^{-s(\lambda t_2)^{\gamma}}g_\alpha(x-n,(at_2)^\alpha s)f_\beta(s,t_1)\,\mathrm{d}s\, \mathrm{d}x\\
		&=\sum_{n=0}^{\infty}e^{-n\eta}\frac{(-\lambda\partial_\lambda)^n}{n!}\int_{0}^{\infty}e^{-s(\lambda t_2)^{\gamma}-(at_2\eta)^\alpha s} f_\beta(s,t_1)\,\mathrm{d}s\\
		&=\sum_{n=0}^{\infty}e^{-n\eta}\frac{(-\lambda\partial_\lambda)^n}{n!}E_{\beta,1}(-t_1^\beta ((\lambda t_2)^{\gamma}-(at_2\eta)^\alpha))\\
		&=e^{-e^{-\eta}\lambda \partial_\lambda}E_{\beta,1}(-t_1^\beta ((\lambda t_2)^{\gamma}-(at_2\eta)^\alpha))\\
		&=E_{\beta,1}(-t_1^\beta ((\lambda(1-e^{-\eta}) t_2)^{\gamma}-(at_2\eta)^\alpha)),
	\end{align*}
	which coincides with \eqref{lapdrf}. This verifies that \eqref{pmfcompdrft} is the density of \eqref{prfcompdrift}.
\end{remark}

\section{L\'evy process time-changed by drifted PRFs}\label{sec5}
In this section, we study some time-changed L\'evy processes where the time-changing components are PRF with and without drift. First, we consider a L\'evy process time-changed by the PRF.

Let $\{Y(t),\ t\ge0\}$ be a one parameter L\'evy process with L\'evy symbol
\begin{equation}\label{lexp}
	\Psi(\xi)=ib\xi+\frac{\xi^2}{2}-\int_{\mathbb{R}-\{0\}}(e^{i\xi x}-1-i\xi x\textbf{1}_{|x|\leq1})\nu(\mathrm{d}x),\ \xi\in\mathbb{R},
\end{equation}
where $b\in\mathbb{R}$ is the drift coefficient and $\nu$ is the L\'evy measure.
Let $\{\mathscr{N}(t_1,t_2),\ (t_1,t_2)\in\mathbb{R}^2_+\}$ be the PRF that is independent of $\{Y(t),\ t\ge0\}$. We consider the following two parameter time-changed process:
\begin{equation}\label{tclp1}
	Y^*(t_1,t_2)\coloneqq Y(\mathscr{N}(t_1,t_2)),\ (t_1,t_2)\in\mathbb{R}^2_+.
\end{equation}
It is a continuous-time real-valued random process. Let us consider a two parameter operator
$T_{t_1,t_2}f(x)=\mathbb{E}f(x+\mathscr{N}(t_1,t_2))$, $(t_1,t_2)\in\mathbb{R}^2_+$ on the space of continuous and bounded real-valued functions.
Then, it can be established that $T_{s_1+t_1,t_2}f(x)=T_{s_1,t_2}T_{t_1,t_2}f(x)$ and $T_{t_1,s_2+t_2}f(x)=T_{t_1,s_2}T_{t_1,t_2}f(x)$, that is, the operator $T_{t_1,t_2}$ has the semigroup property with respect to both coordinates. Also, let us consider the following operator associated with (\ref{tclp1}):
\begin{equation}\label{tcpl1op}
	T^*_{t_1,t_2}f(x)=\mathbb{E}f(x+Y^*(t_1,t_2)),\ (t_1,t_2)\in\mathbb{R}^2_+.
\end{equation}
\begin{proposition}
	If $P_tf(x)=\mathbb{E}f(x+Y(t))$ is the semigroup operator associated with the L\'evy process $\{Y(t),\ t\ge0\}$ with the L\'evy symbol defined in (\ref{lexp}) then the two parameter operator $T^*_{t_1,t_2}$ defined in (\ref{tcpl1op}) has coordinatewise semigroup property. It is given by $T^*_{t_1,t_2}=\mathbb{E}P_{\mathscr{N}(t_1,t_2)}$ for all $(t_1,t_2)\in\mathbb{R}^2_+$.
\end{proposition}
\begin{proof}
	From (\ref{tclp1}) and (\ref{tcpl1op}), we have
	\begin{align*}
		T^*_{t_1,t_2}f(x)&=\mathbb{E}f(x+Y(\mathscr{N}(t_1,t_2)))\\
		&=\sum_{n=0}^{\infty}\mathbb{E}f(x+Y(n))\mathrm{Pr}\{\mathscr{N}(t_1,t_2)=n\}\\
		&=\sum_{n=0}^{\infty}P_nf(x)\mathrm{Pr}\{\mathscr{N}(t_1,t_2)=n\}=\mathbb{E}P_{\mathscr{N}(t_1,t_2)}f(x),
	\end{align*}
	for all continuous and bounded real-valued functions $f(\cdot)$.
	
	Further,  for $s_1\ge0$, we have
	\begin{align*}
		T^*_{t_1+s_1,t_2}f(x)&=e^{-\lambda(s_1+t_1)t_2}\sum_{k=0}^{\infty}P_kf(x)\frac{(\lambda(s_1+t_1)t_2)^k}{k!}\\
		&=e^{-\lambda(s_1+t_1)t_2}\sum_{k=0}^{\infty}P_{k}f(x)(\lambda t_2)^k\sum_{r=0}^{k}\frac{t_1^{k-r}s_1^r}{(k-r)!r!}\\
		&=e^{-\lambda s_1t_2}\sum_{r=0}^{\infty}\frac{(\lambda s_1t_2)^r}{r!}P_r\sum_{k=r}^{\infty}P_{k-r}f(x)e^{-\lambda t_1t_2}\frac{(\lambda t_1t_2)^{k-r}}{(k-r)!}\\
		&=e^{-\lambda s_1t_2}\sum_{r=0}^{\infty}\frac{(\lambda s_1t_2)^r}{r!}P_rT^*_{t_1,t_2}f(x)=T^*_{s_1,t_2}T^*_{t_1,t_2}f(x),
	\end{align*}
	where we have used the semigroup property of $P_t$ to get the penultimate step. Similarly, we can prove that $T^*_{t_1,s_2+t_2}=T^*_{t_1,s_2}T^*_{t_1,t_2}$ for all $s_2\ge0$. This completes the proof.
\end{proof}

\begin{proposition}
The process $\{Y^*(t_1,t_2),\ (t_1,t_2)\in\mathbb{R}^2_+\}$ has stationary rectangular increments.	
\end{proposition}
\begin{proof}
It is sufficient to show that for any $(h,k)\in\mathbb{R}^2_+$, the distribution of 
\begin{equation*}
	\Delta_{h,k}Y^*(s,t)=Y^*(s+h,t+k)-Y^*(s+h,t)-Y^*(s,t+k)+Y^*(s,t)
\end{equation*}
is independent of the choice of $(s,t)$. Let $u(x;\Delta_{h,k}Y^*(s,t))$ and $u(x;Y(t))$ be the density functions of $\Delta_{h,k}Y^*(s,t)$ and $Y(t)$, respectively. Also, let $p(n_1,n_2,n_3,n_4)=\mathrm{Pr}\{\mathscr{N}(s+h,t+k)=n_1,\mathscr{N}(s+h,t)=n_2,\mathscr{N}(s,t+k)=n_3,\mathscr{N}(s,t)=n_4\}$, $n_i\ge0$, $i=1,2,3,4$. Then, from (\ref{tclp1}), we have
\begin{align}
	u(x;\Delta_{h,k}Y^*(s,t))&=\sum_{\substack{n_i\ge0\\i=1,2,3,4}}p(n_1,n_2,n_3,n_4)u(x;Y(n_1)-Y(n_2)-Y(n_3)+Y(n_4))\nonumber\\
	&=\sum_{\substack{n_i\ge0\\i=1,2,3,4\\n_2\ge n_3}}p(n_1,n_2,n_3,n_4)u(x;Y(n_1)-Y(n_2)-Y(n_3)+Y(n_4))\nonumber\\
	&\ \ +\sum_{\substack{n_i\ge0\\i=1,2,3,4\\
	n_2<n_3}}p(n_1,n_2,n_3,n_4)u(x;Y(n_1)-Y(n_2)-Y(n_3)+Y(n_4))\nonumber\\
    &=\sum_{\substack{n_i\ge0\\i=1,2,3,4\\n_2\ge n_3}}p(n_1,n_2,n_3,n_4)\int_{-\infty}^{\infty}u(x+y;Y(n_1)-Y(n_2))u(y;Y(n_3)-Y(n_4))\,\mathrm{d}y\nonumber\\
    &\ \ +\sum_{\substack{n_i\ge0\\i=1,2,3,4\\
    		n_2<n_3}}p(n_1,n_2,n_3,n_4)\int_{-\infty}^{\infty}u(x+y;Y(n_1)-Y(n_3))u(y;Y(n_2)-Y(n_4))\,\mathrm{d}y\label{indin}\\
    &=\sum_{\substack{n_i\ge0\\i=1,2,3,4\\n_2\ge n_3}}p(n_1,n_2,n_3,n_4)\int_{-\infty}^{\infty}u(x+y;Y(n_1-n_2))u(y;Y(n_3-n_4))\,\mathrm{d}y\nonumber\\
    &\ \ +\sum_{\substack{n_i\ge0\\i=1,2,3,4\\
    		n_2<n_3}}p(n_1,n_2,n_3,n_4)\int_{-\infty}^{\infty}u(x+y;Y(n_1-n_3))u(y;Y(n_2-n_4))\,\mathrm{d}y\label{ststin}\\
    &=  \sum_{\substack{n_i\ge0\\i=1,2,3,4\\n_2\ge n_3}}p(n_1,n_2,n_3,n_4)u(x;Y(n_1-n_2)-Y(n_3-n_4))\nonumber\\
    &\ \ +\sum_{\substack{n_i\ge0\\i=1,2,3,4\\n_2<n_3}}p(n_1,n_2,n_3,n_4)u(x;Y(n_1-n_3)-Y(n_2-n_4))\nonumber\\
    &=  \sum_{\substack{n_i\ge0\\i=1,2,3,4\\n_2\ge n_3}}p(n_1,n_2,n_3,n_4)u(x;Y(n_1-n_2-n_3+n_4))\nonumber\\ &\ \ +\sum_{\substack{n_i\ge0\\i=1,2,3,4\\n_2<n_3}}p(n_1,n_2,n_3,n_4)u(x;Y(n_1-n_3-n_2+n_4))\nonumber\\
    &= \sum_{\substack{n_i\ge0\\i=1,2,3,4}}p(n_1,n_2,n_3,n_4)u(x;Y(n_1-n_2-n_3+n_4))\nonumber\\
    &=u(x;Y(\mathscr{N}(s+h,t+k)-\mathscr{N}(s+h,t)-\mathscr{N}(s,t+k)+\mathscr{N}(s,t))),\nonumber
\end{align}
where (\ref{indin}) and (\ref{ststin}) follow by using the independent and stationary increments properties of the L\'evy process $\{Y(t),\ t\ge0\}$, respectively. As the PRF has stationary increments, so does the process (\ref{tclp1}). This completes the proof.
\end{proof}

Next, we consider the  PRF with constant drift as the time-changing component. That is, we consider the following two parameter time-changed process:
\begin{equation*}
	Y(\mathscr{N}(t_1,t_2)+at_1t_2),\ (t_1,t_2)\in\mathbb{R}^2_+,\ a>0,
\end{equation*}
where $\{\mathscr{N}(t_1,t_2)+at_1t_2,\ (t_1,t_2)\in\mathbb{R}^2_+\}$ is defined in (\ref{dprfd=2}) and it is independent of the L\'evy process $\{Y(t),\ t\ge0\}$ whose L\'evy exponent is given in (\ref{lexp}).

\begin{proposition}\label{lprop}
	Let $Y_k$, $k=1,2,\dots$ be independent random variables such that $Y_k\sim Y(1)$. Then, for all $(t_1,t_2)\in\mathbb{R}^2_+$, we have
	\begin{equation}\label{tclpdrp}
		Y(\mathscr{N}(t_1,t_2)+at_1t_2)\overset{d}{=}\sum_{k=0}^{\mathscr{N}(t_1,t_2)}Y_k+Y(at_1t_2),
	\end{equation}
	where $\overset{d}{=}$ denotes the equality in distribution.
\end{proposition}
\begin{proof}
	Let $\Psi(\xi)$ be the L\'evy exponent of $\{Y(t),\ t\ge0\}$ as defined in (\ref{lexp}). Then, the characteristic function of the left hand side of (\ref{tclpdrp}) is given by
	\begin{align*}
		\mathbb{E}\exp(i\xi Y(\mathscr{N}(t_1,t_2)+at_1t_2))&=\mathbb{E}(\mathbb{E}(\exp(i\xi Y(\mathscr{N}(t_1,t_2)+at_1t_2))|\mathscr{N}(t_1,t_2)+at_1t_2))\nonumber\\
		&=\mathbb{E}\exp((\mathscr{N}(t_1,t_2)+at_1t_2)\Psi(\xi))\nonumber\\
		&=e^{at_1t_2\Psi(\xi)}\exp(\lambda t_1t_2(e^{\Psi(\xi)}-1))\\
		&=e^{at_1t_2\Psi(\xi)}\mathbb{E}\exp(\mathscr{N}(t_1,t_2)\Psi(\xi))\\
		&=e^{at_1t_2\Psi(\xi)}\mathbb{E}\bigg(\mathbb{E}\bigg(i\xi\sum_{k=0}^{\mathscr{N}(t_1,t_2)}Y_k\bigg|\mathscr{N}(t_1,t_2)\bigg)\bigg)\\
		&=\mathbb{E}\exp\bigg(i\xi\bigg(\sum_{k=0}^{\mathscr{N}(t_1,t_2)}Y_k+Y(at_1t_2)\bigg)\bigg),
	\end{align*}
which is the characteristic function of the right hand side of (\ref{tclpdrp}). This completes the proof.
\end{proof}	

\begin{remark}
	Let $\{S_\gamma(t), t\ge0\}$, $0<\gamma<1$ and $\{S_\alpha(t),\ t\ge0\}$, $0<\alpha<1$ be independent stable subordinators, and let $\{Y(t),\ t\ge0\}$ be a L\'evy process that is independent of both theses subordinators. Then,
	\begin{equation*}
		Y(\mathscr{N}(S_\gamma(t_1),t_2)+aS_\alpha(t_1)t_2)\overset{d}{=}\sum_{k=0}^{\mathscr{N}(S_\gamma(t_1),t_2)}Y_k+Y(aS_\alpha(t_1)t_2),
	\end{equation*}
	where $Y_k$'s are as defined in Proposition \ref{lprop}. Its proof follows from (\ref{tclpdrp}).
\end{remark}
\section*{Acknowledgement}
The second author thanks Government of India for the grant of Prime Minister's Research Fellowship, ID 1003066.

\end{document}